\documentclass[UKenglish,10pt,final]{article}
\usepackage[
textwidth=124mm,
textheight=186mm,
paperwidth=164mm,
paperheight=226mm,
marginparwidth=15mm,
]{geometry}
\usepackage{enumitem}
\setlist[enumerate,1]{label=(\arabic*),font=\normalfont,align=left,leftmargin=0pt,labelindent=0pt,listparindent=\parindent,labelwidth=0pt,itemindent=!,topsep=3pt,parsep=0pt,itemsep=3pt,start=1}
\setlist[enumerate,2]{label=(\alph*),font=\normalfont,labelindent=0pt,leftmargin=1cm,start=1}
\setlist[itemize]{labelindent=*,leftmargin=*,topsep=5pt,itemsep=3pt}
\setlist[description]{labelindent=*,leftmargin=*,itemindent=-1 em}

\usepackage[utf8]{inputenc}
\usepackage[T1]{fontenc}
\usepackage{hyperref}
\usepackage{babel}
\usepackage{ifdraft}
\usepackage{amsmath}
\usepackage{xspace}
\usepackage{amssymb}
\usepackage{mathtools}
\usepackage{amsthm}
\usepackage{mathrsfs}
\usepackage{thmtools}
\usepackage[notcite,notref]{showkeys}%
\usepackage{etoolbox}
\usepackage{xstring}
\usepackage{seqsplit}
\usepackage{etoolbox}
\usepackage{algpseudocode}
\usepackage{proof}
\usepackage{dsfont}
\usepackage{cite}

\makeatletter

\renewcommand*\showkeyslabelformat[1]{%
\@ifundefined{hideNextShowKeysLabel}{%
\noexpandarg%
\StrSubstitute{#1}{ }{\textvisiblespace}[\TEMP]%
\parbox[t]{\marginparwidth}{\raggedright\normalfont\small\ttfamily\(\{\){\color{red!50!black}\expandafter\seqsplit\expandafter{\TEMP}}\(\}\)}%
}{}
}
\makeatother

\usepackage{tikz}
\usetikzlibrary{cd,calc}

\usepackage[all]{xy}
\SelectTips{cm}{}

\newdir{ >}{{}*!/-5pt/\dir{>}}
\newdir{ (}{{}*!/-5pt/\dir^{(}}

\numberwithin{equation}{section}

\newcommand{\ownthmSpaceAbove}{5pt}
\newcommand{\ownthmSpaceBelow}{5pt}
\newcommand{\resetCurThmBraces}{%
\gdef\curThmBraceOpen{(}%
\gdef\curThmBraceClose{)}}
\resetCurThmBraces

\declaretheoremstyle[
    spaceabove=\ownthmSpaceAbove,
    spacebelow=\ownthmSpaceBelow,
    headpunct=.,
    postheadspace=.5em,
    notebraces={\curThmBraceOpen}{\curThmBraceClose},
    postheadhook={\resetCurThmBraces},
]{definition}
\declaretheoremstyle[
    style=definition,
    bodyfont=\itshape,
    notebraces={\curThmBraceOpen}{\curThmBraceClose},
    postheadhook={\resetCurThmBraces},
]{theorem}
\makeatletter
\declaretheoremstyle[
    style=definition,
    preheadhook=\renewcommand\@upn{},
    headpunct=.,
    headfont=\it,
]{remark}
\makeatother
\declaretheorem[numberwithin=section,style=definition]{definition}
\declaretheorem[style=definition,sibling=definition]{construction}

\declaretheorem[style=definition,sibling=definition]{example}
\declaretheorem[style=definition,sibling=definition]{notation}
\declaretheorem[style=definition,sibling=definition]{assumption}

\declaretheorem[style=theorem,sibling=definition]{theorem}
\declaretheorem[style=theorem,sibling=definition]{corollary}

\declaretheorem[style=theorem,sibling=definition]{lemma}
\declaretheorem[style=theorem,sibling=definition]{proposition}

\declaretheorem[style=definition,sibling=definition]{remark}

\newcommand{\C}{\ensuremath{\mathscr{C}\xspace}}
\newcommand{\D}{\ensuremath{\mathscr{D}\xspace}}
\newcommand{\E}{\ensuremath{\mathscr{E}\xspace}}
\newcommand{\M}{\ensuremath{\mathscr{M}\xspace}}
\newcommand{\N}{\ensuremath{\mathds{N}\xspace}}

\newcommand{\pow}{\ensuremath{\mathcal{P}\xspace}}
\newcommand{\powf}{\ensuremath{\mathcal{P}_{\mathsf{f}}\xspace}}
\newcommand{\powufs}{\ensuremath{\mathcal{P}_{\mathsf{ufs}}\xspace}}
\newcommand{\supp}{\ensuremath{\mathsf{supp}}\xspace}
\newcommand{\A}{\ensuremath{\mathds{A}\xspace}}
\newcommand{\B}{\ensuremath{\mathcal{B}\xspace}}
\newcommand{\Dist}{\ensuremath{\mathcal{D}\xspace}}

\newcommand{\oper}[1]{\ensuremath{\operatorname{\textnormal{\textsf{#1}}}\xspace}}
\newcommand{\id}{\ensuremath{\operatorname{\textsf{id}}\xspace}}
\newcommand{\inl}{\ensuremath{\operatorname{\textsf{inl}}\xspace}}
\newcommand{\inr}{\ensuremath{\operatorname{\textsf{inr}}\xspace}}
\newcommand{\inj}{\ensuremath{\operatorname{\textsf{in}}\xspace}}
\newcommand{\Path}{\ensuremath{\mathscr{P}}}
\newcommand{\Model}{\ensuremath{\mathscr{C}}}
\newcommand{\Set}{\ensuremath{\mathsf{Set}}\xspace}
\newcommand{\Nom}{\ensuremath{\mathsf{Nom}}}
\newcommand{\pure}{\ensuremath{\mathsf{p}}}
\newcommand{\Kl}{\ensuremath{\mathcal{K}\!\ell}}
\newcommand{\Rel}{\ensuremath{\mathsf{Rel}}}
\newcommand{\Coalg}{\oper{Coalg}\xspace}

\newcommand{\arrow}{\rightarrow}

\ifdraft{
\input{defxyar}
}{}

\newsavebox{\mypullbackcorner}%
\sbox{\mypullbackcorner}{%
\begin{tikzpicture}
    \draw[-] (0,0) -- (.5em,.5em) -- (0,1em);
\end{tikzpicture}%
}
\newcommand{\pullbackangle}[2][]{\arrow[phantom,to path={
                     -- ($ (\tikztostart)!1cm!#2:([xshift=8cm]\tikztostart) $)
                        node[anchor=west,pos=0.0,rotate=#2,
                        inner xsep = 0]
                        {\begin{tikzpicture}[minimum
                        height=1mm,baseline=0,#1]
    \draw[-] (0,0) -- (.5em,.5em) -- (0,1em);
                        \end{tikzpicture}}}]{}}

\tikzstyle{shiftarr}=[
        rounded corners,%
        to path={--([#1]\tikztostart.center)
                     -- ([#1]\tikztotarget.center) \tikztonodes
                     -- (\tikztotarget)},
]

\newcommand{\takeout}[1]{\empty}

\newcommand{\ext}[1]{#1^\ast}

\newcommand{\sym}{\ensuremath{\mathfrak{S}}\xspace}
\newcommand{\perms}{\ensuremath{\sym_\mathsf{f}}\xspace}

\renewcommand{\S}{\mathds{S}}

\newcommand{\kth}{\ensuremath{k^{\text{th}}}\xspace}

\newcommand{\sub}{\mathop{\mathsf{Sub}}}

\newcommand{\set}[1]{\{#1\}}

\newcommand{\hookto}{\hookrightarrow}
\newcommand{\subto}{\hookto}
\newcommand{\epito}{\twoheadrightarrow}

\newcommand{\monoto}{\rightarrowtail}

\newsavebox{\kleisliarrow}
\savebox{\kleisliarrow}{%
\begin{tikzpicture}[
      baseline=(arrow.base),
      inner sep=8mm,
      outer sep=0mm,
      ]
      \node[draw=none,
      anchor=base,
      inner sep=0,
      outer sep=0,
      ] (arrow) {$\longrightarrow$};
    \draw[fill=white] ($ (arrow.south) !.68! (arrow.north)$) circle (0.15em);
  \end{tikzpicture}}

\newcommand{\kleislito}{\ensuremath{\mathbin{\usebox{\kleisliarrow}}}}
\newsavebox{\kleislidot}
\savebox{\kleislidot}{%
\begin{tikzpicture}[baseline=0pt,outer sep=0pt]
    \draw[fill=white] (0,0) circle (2pt);
  \end{tikzpicture}}

\tikzstyle{kleisli}=[
"{\usebox{\kleislidot}}"{red,anchor=center,font=\normalsize,pos=0.5},
outer sep = 1pt, 
]

\newcommand{\pasttime}{{-}\hspace{1pt}\llap{$\bigcirc$}}
\newcommand{\nexttime}{\bigcirc}
\title{A Coalgebraic View on Reachability}
\author{
  Thorsten Wißmann\footnote{Email:
    \texttt{\{thorsten.wissmann,stefan.milius\}@fau.de} Address:
    Friedrich-Alexander-Universit\"at Erlangen-N\"urnberg, Germany.
    The authors were supported by the DFG project MI~717/5-1. The first author expresses his
    gratitude for having been invited to Tokyo, which initiated the present work.
  },
  Stefan Milius\rlap{,}${}^*$\\
  Shin-ya Katsumata\rlap{,}${}^\dagger$ and
  Jérémy Dubut\footnote{Email: \texttt{\{dubut,s-katsumata\}@nii.ac.jp}
    Address: National Institute of Informatics, 2-1-2 Hitotsubashi, Chiyoda-ku,
    Tokyo, 101-8430, Japan.
    The authors were supported by ERATO HASUO Metamathematics for Systems Design Project (No. JPMJER1603), JST
  }${}^{\ ,}$\footnote{Japanese-French Laboratory for Informatics, Tokyo, Japan}
}

\makeatletter
\hypersetup{
    final,
    hidelinks,
    urlcolor=black,
    pdftitle={\@title},
    pdfauthor={\@author},
    pdfkeywords={},
    pdfduplex={DuplexFlipLongEdge},
}
\makeatother

\usepackage[footnote,marginclue,nomargin]{fixme}
\FXRegisterAuthor{sm}{asm}{SM} 
\FXRegisterAuthor{tw}{atw}{TW} 
\FXRegisterAuthor{jd}{ajd}{JD} 
\FXRegisterAuthor{sk}{ask}{SK} 

\begin{document}
\maketitle
\centerline{\itshape To the memory of V\v{e}ra Trnkov\'a}

\begin{abstract}%
  \takeout{
  Breadth-first search is a standard algorithm to traverse all nodes in a graph
  structure from a given initial node. Consequently, only the reachable nodes
  are traversed, so it can be used to compute the reachable part of a graph. We
  formulate breadth-first search as a simple categorical construction on a
  pointed coalgebras for an endofunctor $F$. We prove that the result is a
  reachable coalgebra in an abstract sense i.e.~that it does not have a proper
  subcoalgebra. If the functor $F$ preserves preimages, then the
  construction yields a coreflection of a given coalgebra into the
  category of reachable coalgebras.
}
\noindent
Coalgebras for an endofunctor provide a category-theoretic framework
for modeling a wide range of state-based systems of various types. We
provide an iterative construction of the reachable part of a given
pointed coalgebra that is inspired by and resembles the standard
breadth-first search procedure to compute the reachable part of a
graph. We also study coalgebras in Kleisli categories: for a functor
extending a functor on the base category, we show that the reachable
part of a given pointed coalgebra can be computed in that base
category.
\end{abstract}

\section{Introduction}

Coalgebras provide a convenient category theoretic framework in which
to model state-based systems and automata whose transition type is
described by an endo\-functor. For example, classical deterministic and
non-deterministic automata, labelled transition systems as well as
their weighted and probabilistic variants arise as instances of
coalgebras. 

A key notion in the theories of state-based systems of various types
is reachability, i.e.~the construction of a subsystem of a given
system containing precisely those states that can be reached from (a
set of) initial states along a path in the transition graph of the
system. For example, in automata theory, computing the reachable part
of a given deterministic automaton is the first step in every
minimization procedure. It is well-known that reachability has a
simple formulation on the level of coalgebras. In fact, a pointed
coalgebra, i.e.~one with a given initial state, is called reachable if
it does not contain any proper subcoalgebra containing the initial
state~\cite{amms13}. Moreover, for a functor preserving intersections,
the reachable part of a given pointed coalgebra is obtained by taking
the intersection of all the subcoalgebras containing the initial
state. The purpose of the present paper is a more thorough study of
reachable coalgebras and, in particular, a new iterative construction
of the reachable part of a given pointed coalgebra.

After recalling some preliminaries in \autoref{S:preach}, we discuss
some background material on endofunctors on $\Set$ preserving
intersections in \autoref{S:inter} and on the canonical graph of a
coalgebra in $\Set$ in \autoref{S:cangraphs}.

In \autoref{S:iter} we present a new iterative construction of the reachable
part of a given pointed coalgebra that is inspired by
and closely resembles the standard breadth-first search in graphs. Our
construction works for coalgebras over every well-powered category
$\C$ having coproducts and a factorization system $(\E,\M)$, where
$\M$ consists of monomorphisms. Moreover, the coalgebraic type functor
$F\colon \C \to \C$ is assumed to have \emph{least bounds}, a notion
previously introduced by Block~\cite{Block12}. Extending a result by
Gumm~\cite{gumm05filter} for set functors, we prove in
\autoref{P:leastfact} that a functor has least bounds if and only if
it preserves intersections. Moreover, this is equivalent to the
existence of a left-adjoint to the operator
$\nexttime_f\colon \sub(Y) \to \sub(X)$ for every $f\colon X \to FY$,
which assigns to every subobject $m$ of $Y$ the pullback of $Fm$ along
$f$. Note that, for a coalgebra $c\colon C \to FC$, this operator is
Jacobs' ``next time'' operator~\cite{Jacobs02}. In our iterative
construction of the reachable part we use its left-adjoint
$\pasttime_c$ on $\sub(C)$, which corresponds to the ``previous time''
operator of classical linear temporal logic~\cite{MP92}. In fact, we
consider a coalgebra $c\colon C \to FC$ together with an
\emph{$I$-pointing}, i.e.~a morphism $i_C\colon I \to C$, where $I$ is
some object, and we prove in our main result \autoref{constrCorrect}
that the reachable part of the given $I$-pointed coalgebra is given by
the union of all $\pasttime^k(m_0)$, where $m_0$ is given by
$(\E,\M)$-factorizing the given $I$-pointing $i_C$. Moreover, we prove
that, whenever $F$ preserves inverse images, the reachable part is a
coreflection of $(C,c,i_C)$ into the category of $I$-pointed reachable
$F$-coalgebras (\autoref{thm:coreflect}). We also show that for an
$I$-pointed coalgebra in $\Set$ the above iterative construction of
the reachable part can be performed as a standard breadth-first search
on the canonical graph (\autoref{C:set}).

Finally, we study in \autoref{S:kleisli} coalgebras for a functor
$\bar F$ on a Kleisli category over $\C$, which is an extension of an
endofunctor $F$ on $\C$. Here we show that the reachable part of a given
$I$-pointed $\bar F$-coalgebra can be constructed as the reachable part 
of a related coalgebra in $\C$.

\enlargethispage{10pt}
\paragraph{Dedication.} We would like to dedicate this paper to the memory of
V\v{e}ra Trnkov\'a. Her research, especially her foundational results
of the late 1960s and early 1970s on set functors, are still
continuing to have considerable impact, in particular for the theory of
coalgebras. In addition, her work on Kleisli categories and lifting
functors to categories of relations is of basic importance for work on
coalgebraic logic. We make use of Trnkov\'a's careful research on
properties of set functors in \autoref{S:inter}.

\paragraph{Related work.} Our results are based on the notion of
reachable coalgebras introduced by Ad\'amek et al.~\cite{amms13}. Our
construction of the reachable part appears in work by Wißmann, Dubut,
\text{Katsumata}, and Hasuo~\cite{pathCatFree} (see Lemma~A.5 of the
full version), where it is used as an
auxiliary construction in order to give a
characterization of the reachability of a coalgebras in terms of
paths~\cite[Section~3.5]{pathCatFree}. However, that work does not
connect the construction with the ``previous time'' operator.

The ``previous time'' operator is also studied by Barlocco, Kupke, and
Rot~\cite{BarloccoEA19}. They work with a complete and well-powered
category $\C$, and, like us, they show
that the reachable part of a given pointed coalgebra can be obtained
by an iterative construction using the ``previous time'' operator. Their
results were obtained independently and almost at the same time as
ours.

\section{Pointed and Reachable Coalgebras}
\label{S:preach}

In this section we recall some preliminaries on pointed and reachable
coalgebras for an endofunctor.  A coalgebra for an endofunctor
$F\colon \C \to \C$ (or $F$-coalgebra, for short) is a pair $(C, c)$
where $C$ is an object of $\C$ called the \emph{carrier} of the
coalgebra and $c\colon C \to FC$ a morphism called the
\emph{structure} of the coalgebra. A \emph{coalgebra homomorphism}
$h\colon (C,c) \to (D,d)$ is a morphism $h\colon C \to D$ of $\C$ that
commutes with the structures on $C$ and $D$, i.e.~the following square
commutes:
\[
  \begin{tikzcd}
    C \arrow{r}{c} \arrow{d}[swap]{h} & FC \arrow{d}{Fh} \\
    D \arrow{r}{d} & FD
  \end{tikzcd}
\]
\begin{definition}
  Given an endofunctor $F\colon \C\to\C$ and an object $I$ of $\C$, an
  \emph{$I$-pointed $F$-coalgebra} is a triple $(C,c,i_C)$ where
  $(C,c)$ is an $F$-coalgebra and $i_C\colon I \to C$ a morphism of
  $\C$. A homomorphism of $I$-pointed coalgebras from $(C,c,i_C)$ to
  $(D,d,i_D)$ is a coalgebra homomorphism $h\colon (C,c) \to (D,d)$
  preserving the pointings, i.e.~$h \cdot i_C = i_D$. We denote by
  \[
    \Coalg_I(F)
  \]
  the category of $I$-pointed $F$-coalgebra and their homomorphisms.
\end{definition}
\begin{example}\label{E:coalg}
  Pointed coalgebras allow to capture many kinds of state-based
  systems categorically. We just recall a couple of examples; for further
  examples, see e.g.~\cite{rutten00}. 
  \begin{enumerate}
  \item Deterministic automata are $5$-tuples
    $(S, \Sigma, \delta, s_0, F)$, with a set $S$ of states, an input
    alphabet $\Sigma$, a next-state function
    $\delta\colon S \times \Sigma \to S$, an initial state $s_0 \in S$
    and a set $F \subseteq S$ of final states. Here we fix the input
    alphabet $\Sigma$. Representing the subset $F$ by its
    characteristic function $f\colon S \to \{0,1\}$, and currying $\delta$
    we see that a deterministic automaton is, equivalently, a pointed
    coalgebra for $FX = \{0,1\} \times X^\Sigma$ on $\Set$ with the pointing
    $s_0\colon 1 \to S$ given by the initial state.

  \item\label{E:coalg:2} Non-deterministic automata are similar to deterministic ones,
    except that in lieu of a next-state function one has a next-state
    relation $\delta \subseteq S \times \Sigma \times S$ and a set of
    initial states $I \subseteq S$. These data can be represented as
    two functions $i\colon 1 \to \pow S$ and
    $c\colon S \to \pow(1 + \Sigma \times S)$, where $\pow$ denotes
    the power-set. That means that a non-deterministic automaton is,
    equivalently, a coalgebra for the functor
    $FX = 1 + \Sigma \times X$ on the Kleisli category of the monad $\pow$,
    i.e.~the category $\Rel$ of sets and relations.

  \item Pointed graphs are, equivalently, coalgebras for the power-set functor
    $\pow\colon \Set \to\Set$. Indeed, a pointed coalgebra
    \[
      1 \xrightarrow{v_0} V \xrightarrow{e} \pow V
    \]
    consists of a set of vertices $V$ with directed edges given by a binary
    relation, represented by $e$, and a distinguished node
    $v_0 \in S$.
  \item The category of nominal sets provides a framework where
    freshness of names or resources in systems can be modelled or
    where systems can store values from infinite data types. We
    briefly recall the definition of the category $\Nom$ of nominal
    sets (see e.g.~Pitts~\cite{Pitts13}). We fix a countably infinite
    set $\A$ of \emph{atomic names}. Let $\perms(\A)$ denote the group
    of all finite permutations on $\A$ (which is generated by all
    transpositions $(a\,b)$ for $a,b \in \A$). Let $X$ be a set with
    an action of this group, denoted by $\pi \cdot x$ for a finite
    permutation $\pi$ and $x \in X$. A subset $A \subseteq \A$ is
    called a \emph{support} of an element $x \in X$ provided that
    every permutation $\pi \in \perms(\A)$ that fixes all elements of
    $A$ also fixes $x$:
    \[
      \forall \pi\in \perms(\A)\colon \text{$\pi(a) = a$ for all $a \in A \implies \pi \cdot x = x$}. 
    \]
    A \emph{nominal set} is a set with an action of the group
    $\perms(\A)$ such that every element has a finite support. The
    category $\Nom$ is formed by nominal sets and \emph{equivariant
      maps}, i.e.~maps preserving the given group action.  Each
    nominal set $X$ is thus equipped with an equivariant map
    $\supp\colon X\to \powf(\A)$ that assigns to each element its least
    support. For example, the set of terms of the $\lambda$-calculus
    modulo renaming of bound variables is a nominal set, where
    the least support of a $\lambda$-term is the set of its free
    variables. Variable binding can be modelled by the \emph{binding}
    functor on $\Nom$. This functor maps a nominal set $X$ to the
    nominal set $[\A](X) = (\A \times X)/\mathord{\sim}$ where
    $(a,x)\sim(b,y)$ iff $(c\,a)\cdot x=(c\,b)\cdot y$ for any \emph{fresh}
    $c$, i.e.  $c \not\in \supp(x) \cap \supp(y)$. That means that
    $\sim$ abstracts $\alpha$-equivalence known from calculi with name
    binding such as the $\lambda$-calculus. In fact, the set of
    $\lambda$-expressions modulo $\alpha$-equivalence is the initial
    algebra for the endofunctor $FX = \A + X\times X + [\A]X$ on
    $\Nom$~\cite{gp99}.
    %
    %
    Coalgebras for functors on $\Nom$ have been studied
    e.g.~in~\cite{KurzEA13,mw15,msw16}.
    
    There are a number of different notions of automata featuring a
    nominal set of states and which process words over the infinite
    input alphabet $\A$. One example of a coalgebraic notion of
    automata are regular nondeterministic nominal automata
    (RNNA)~\cite{skmw17}; they are precisely the coalgebras for the
    functor on nominal sets given by
    \[
      FX = 2\times \powufs(\A\times X) \times \powufs([\A]X),
    \]
    where $\powufs$ is a variant of the finite power-set functor on
    $\Nom$ -- it maps a nominal set $X$ to the nominal set of all of
    its \emph{uniformly supported} subsets $S$, i.e.~$S$ is an
    equivariant subset of $X$ such that $\bigcup_{x \in S} \supp(x)$
    is finite.

    Intuitively, in a coalgebra $C \to 2\times \powufs(\A\times C) \times
    \powufs([\A]C)$, $2$ marks whether a state is final; $\powufs([\A]C)$ is the
    set of \emph{binding transitions} from the state $x$, i.e.~where the input
    letter is stored for later use; and $\powufs(\A\times C)$ is the set of
    transitions that compare the input letter to a previously stored one.
    Let $\A^{\# n}$ be the nominal set of $n$-tuples with distinct components,
    i.e.
    \[
      \A^{\# n} = \set{(a_1, \ldots, a_n) \in \A^n \mid |\set{a_1, \ldots
          a_n}| = n}.
    \]
    \sloppypar\noindent
    Then $\A^{\# n}$-pointed RNNAs accept nominal languages,
    i.e.~finitely supported maps \mbox{$L\colon \A^*\to 2$}, whose support has a cardinality
    of at most $n$~\cite[Corollary 5.5]{skmw17}, under both language semantics
    considered in~\emph{op.~cit}. Note that it is important not to
    restrict $I$ to be the terminal object; in fact, that would
    restrict initial objects to have empty support, which may not be
    desirable in applications. 

  \item An alternative approach to bisimulation of transition systems via so
    called \emph{open maps} was introduced by Joyal, Nielsen, and
    Winskel~\cite{joyal96}. There, one considers functors of type $J\colon \Path
    \to \Model$\smnote{I think we should change $\Model$ to $\C$
      because we use $\M$ for a class from a factorization system in
      the rest of the paper. TW: I've changed the macro definition and I think
      it's fine.} from a small category $\Path$ of
    ``paths'' or ``linear systems'' to the category $\Model$ of ``all
    systems'' under consideration. This functor $J$ defines a notion of
    \emph{open map}. We do not recall the definition as it is irrelevant here;
    for details and the definition of open map see \emph{op.~cit}. The objects
    in $\Model$ are usually defined as systems with an initial state, and
    morphisms in $\Model$ are maps between systems that preserve (but not
    necessarily reflect) outgoing transitions of states, whereas the open maps
    in $\Model$ are morphisms that do reflect the outgoing transitions of states
    that are in the reachable part of the system. Let $|\Path |$ denote the set
    of objects of $\Path $. It was shown by Lasota~\cite{lasota02} that the
    canonical functor $\Model(J(-),(-))\colon \Model\to \Set^{|\Path |}$ sends
    open maps in $\Model$ to $F$-coalgebra homomorphisms for the following
    functor
    \[
      F\colon \Set^{|\Path |}
      \to  \Set^{|\Path |}
      \quad\text{given by}\quad
      (X_P)_{P\in\Path }
      \mapsto
      \big(
      \prod_{\mathclap{Q\in |\Path |}} \pow(X_Q)^{\Path (P,Q)}
      \big)_{P\in \Path}\,.
    \]
    If $\Path $ has an initial object $0_\Path $ that is preserved by $J$, then
    the subcategory of $\Model$ formed by all open maps embeds into the
    category of $I$-pointed $F$-coalgebras~\cite{pathCatFree}, where
    \[
      I \in \Set^{|\Path |}\qquad\text{with}\qquad
      I_P = \begin{cases}
        1 &\text{if }P = 0_{\Path } \\
        \emptyset &\text{otherwise.} \\
      \end{cases}
    \]
    Note that once again $I$ is not the terminal object of $\Set^{|\Path |}$.
  \end{enumerate}
\end{example}

Our overall setting is that of a
category $\C$ equipped with a \emph{factorization system} $(\E, \M)$,
i.e.~(1)~$\E$ and $\M$ are classes of morphisms of $\C$ that are closed
under composition with isomorphisms, (2)~every morphism $f$ of $\C$
has a factorization $f = m \cdot e$ with $m \in \M$ and $e \in \E$,
and~(3)~the following \emph{unique diagonal fill-in} property holds:
for every commutative square
\[
  \begin{tikzcd}
    A \arrow{r}{e} \arrow{d}[swap]{f} & B \arrow{d}{g} \arrow[dashed,swap]{dl}{d}\\
    C \arrow{r}{m} & D
  \end{tikzcd}
\]
with $e \in \E$ and $m \in \M$ there exists a unique morphism
$d\colon B \to C$ such that $m \cdot d = g$ and $d \cdot e = f$. We
will denote morphisms in $\M$ by $\monoto$ and those in $\E$ by
$\epito$. While not necessary, in typical examples $\E$ is a class of
epimorphisms, and $\M$ is a class of monomorphisms: (regular epi,
mono) in regular categories, (epi, strong mono) in quasitoposes, (epi,
mono) in toposes, etc.

We shall later only assume that $\M$ is a class of monomorphisms.  Whenever
we speak of a \emph{subobject} of some object $X$ we mean one that is
represented by a morphism $m\colon S \monoto X$ in $\M$. Moreover, we
shall speak of ``the subobject $m$'', i.e.~we use representatives to refer
to subobjects. The subobjects of an object $X$ form a partially
ordered class
\[
  \sub(X)
\]
in the usual way: for $m\colon S \monoto X$ and $m'\colon S'\monoto X$
we have $m \leq m'$ if there exists $i\colon S\to S'$ with
$m' \cdot i = m$.

\begin{remark} \label{rem:EM}
  $(\E,\M)$-factorization systems have many properties known from surjective
  and injective maps on \Set
  (see~\cite[Chapter~14]{ahs09}):\smnote{pointing to a different
    Proposition in the same chapter of the same source looks a bit
    funny; we shouldn't underestimate our readers.}
  \begin{enumerate}
  \item $\E\cap \M$ is the class of isomorphisms of $\C$.
  \item $\M$ is stable under pullbacks.
  \item\label{rem:EM:3} If $f\cdot g\in \M$ and $f\in \M$, then $g\in \M$. 
  \item $\E$ and $\M$ are closed under composition.
  \end{enumerate}
\end{remark}

\begin{remark}\label{rem:subEM}
  \begin{enumerate}
  \item\sloppypar \emph{Subcoalgebras} of pointed coalgebras are understood to
    be formed w.r.t.~the class $\M$, i.e.~a subcoalgebra is
    represented by a homomorphism \linebreak
    $m\colon (S,s,i_S) \monoto (C,c,i_C)$ with $m \in \M$. Similarly, a
    \emph{quotient coalgebra} is represented by a 
    homomorphism $q\colon (C, c, i_C) \epito (Q, q, i_Q)$ with $q
    \in \E$.
  \item\label{rem:subEM:2} Suppose that $F\colon \C \to \C$ preserves
    $\M$-morphisms, i.e.~$Fm \in \M$ for every $m \in \M$. Then the
    factorization system $(\E, \M)$ lifts to $\Coalg_I(F)$ as follows. For every
    homomorphism $h\colon (C,c,i_C) \to (D,d,i_D)$ one takes its
    factorization $h = m \cdot e$ in $\C$ and then obtains a unique coalgebra
    structure such that $e$ and $m$ are homomorphisms of $I$-pointed
    coalgebras using the unique diagonal fill-in property:

    \[
      \begin{tikzcd}
        &
        C \arrow{r}{c} \arrow[->>]{d}[swap]{e} & FC \arrow{d}{Fe} \\
        I
        \arrow{ur}{i_C}
        \arrow{dr}[swap]{i_D}
        \arrow{r}{e\cdot i_C}
        & X \arrow[dashed]{r} \arrow[>->]{d}[swap]{m} & FX \arrow[>->]{d}{Fm} \\
        &
        D \arrow{r}{d} & FD
      \end{tikzcd}
    \]
  \end{enumerate}
\end{remark}

\begin{definition}[Reachable coalgebra~\cite{amms13}]
  An $I$-pointed coalgebra $(C,c,i_0)$ is called \emph{reachable} if
  it has no proper pointed subcoalgebra, i.e.~every homomorphism
  $m\colon (C', c', i_{C'}) \monoto (C,c,i_C)$ of
  $I$-pointed coalgebras with $m \in \M$ is an isomorphism.
\end{definition}
\begin{remark}\label{R:triv}
  When $I = 0$ is the initial object and $\M$ is a class of
  monomorphisms, then a coalgebra is reachable if and only if it is a
  quotient coalgebra of $(0, u, \id_0)$ where $u\colon 0 \to F0$ is
  the unique morphism. Indeed, suppose that $(C,c, i_C)$ is reachable, let
  $h\colon (0,u,\id_0) \to (C,c,i_C)$ be the unique 
  homomorphism, and take the $(\E,\M)$-factorization $h = m\cdot
  e$. Then $m$ represents an $I$-pointed subcoalgebra of $(C,c,i_C)$ and
  thus is an isomorphism.

  \sloppypar
  Conversely, if $e\colon (0,u,\id_0) \epito (C,c,i_C)$ is a quotient
  coalgebra and $m\colon (S,s,i_S) \monoto (C,c,i_C)$ is any
  subcoalgebra then $e = m \cdot h$ where $h\colon 0 \to S$ is the
  unique morphism. By the unique diagonalization property, we obtain
  $d\colon C \to S$ such that $m \cdot d = \id_C$. Thus $m$ is a split
  epimorphism and a monomorphism, whence an isomorphism. Consequently,
  $(C,c, i_C)$ is reachable.

  Finally, it follows that if the unique morphisms $0 \to X$ are in $\M$ for
  every object $X$, then $(0, u, \id_0)$ is the only reachable
  $0$-pointed coalgebra.
\end{remark}
\begin{example}\label{E:reach}
  \begin{enumerate}
  \item\label{E:reach:1} For a pointed graph, reachability is clearly
    the usual graph theoretic concept: a pointed coalgebra $(V,a,v_0)$
    for $\pow$ is reachable if and only if every of its nodes can be
    reached by a directed path from the distinguished node~$v_0$.
  \item\label{E:reach:2} A deterministic automaton regarded as a
    pointed coalgebra for $FX = \{0,1\} \times X^\Sigma$ on $\Set$ is
    reachable if and only if every of its states is reachable in
    finitely many steps from its initial state. This is not difficult
    to see directly, but it follows immediately from
    \autoref{T:reach}.
  \end{enumerate}
\end{example}

\section{Functors preserving intersections }
\label{S:inter}
We shall see in \autoref{S:iter} that the central assumption for our
constructions of the reachable part of a pointed $F$-coalgebra is
equivalent to the functor $F$ preserving intersections. For set
functors we discuss this condition in the present section. Indeed, it
is an extremely mild condition satisfied by many set functors of
interest:

\begin{example}
  The collection of set functors which preserve intersections is
  closed under products, coproducts, and composition.
  Consequently, every polynomial endofunctor on $\Set$ preserves 
  intersections. Moreover it is easy to see that the power set functor
  $\pow$, the bag functor $\B$ mapping every set $X$ to the set of
  finite multisets on $X$, as well as the functor $\D$ mapping $X$ to
  the set of (countably supported) probability measures on $X$ preserve
  intersections.
\end{example}

Among the finitary set functors ``essentially'' all functors preserve
intersections. This follows from the results of Trnkov\'a on set
functors as we shall now explain. First, recall that a functor is
called \emph{finitary}, if it preserves filtered colimits. For a set functor $F$ this
is equivalent to being \emph{finitely
  bounded}~\cite[Corollary~3.13]{amsw19}, which is the following
condition: for every element $x \in FX$ there exists a finite subset
$M \subseteq X$ such that $x \in Fi[FM]$, where $i\colon M \subto X$
is the inclusion map.

Secondly, as shown by Trnkov\'a~\cite{trnkova69}, every set functor
preserves finite non-empty intersections. Moreover, she proved that
one can turn every set functor into one that preserves all finite
intersections by a simple modification at the empty set:

\begin{proposition}[Trnkov\'a~\cite{trnkova71}]\label{P:Tr}
  For every set functor $F$ there exists an essentially unique set
  functor $\bar F$ which coincides with $F$ on non-empty sets and
  functions and preserves finite intersections (whence monomorphisms).
\end{proposition}
For the proof see~\cite[Propositions III.5 and II.4]{trnkova71}; for a
more direct proof see Ad\'amek and Trnkov\'a~\cite[III.4.5]{AT90}.
We call the functor $\bar F$ the \emph{Trnkov\'a hull} of $F$.

\begin{remark}\label{R:inter}
  \begin{enumerate}
  \item In fact, Trnkov\'a gave a construction of $\bar F$: she
    defined $\bar F \emptyset$ as the set of all natural
    transformations $C_{01} \to F$, where $C_{01}$ is the set functor
    with $C_{01} \emptyset = \emptyset$ and $C_{01} X =1$ for all
    non-empty sets $X$, and $\bar F e$, for the empty map
    $e\colon \emptyset \to X$ with $X\neq \emptyset$, maps a natural
    transformation $\tau\colon C_{01} \to F$ to the element given by
    $\tau_X\colon 1\to FX$.
  \item There is also a different construction of $\bar F$ due to
    Barr~\cite{Barr93}: consider the two functions
    $t, f\colon 1 \subto 2$. Their intersection is the empty function
    $e\colon \emptyset \to 1$. Since $\bar F$ must preserve this
    intersection it follows that $\bar Fe$ is monic and forms (not
    only a pullback but also) an equalizer of $\bar F t = Ft$ and
    $\bar F f = Ff$. Thus $\bar F$ must be defined on $\emptyset$ (and
    e) as the equalizer
    \[
      \begin{tikzcd}
        \bar F\emptyset
        \arrow{r}{\bar F e}
        &
        \bar F1
        =
        F1
        \arrow[shift left=1]{r}{Ft}
        \arrow[shift right=1]{r}[swap]{Ff}
        & F2,
      \end{tikzcd}
    \]
    and on all non-empty functions $f$, one defines $\bar F f = Ff$.

  \item Trnkov\'a proved that $\bar F$ defines a set functor
    preserving finite intersections.  From the proof in
    \emph{op.~cit.} it also follows that if $F$ is finitary, so if
    $\bar F$. 
  \item Furthermore, $\bar F$ is a reflection of $F$ into the full
    subcategory of the category of all endofunctors on $\Set$ given by
    those endofunctors preserving finite intersections. That means
    there is a natural transformation $r\colon F \to \bar F$ such that for
    every natural transformation $s\colon F \to G$ where $G:\Set \to \Set$
    preserves finite intersections there exists a unique natural
    transformation $s^\sharp\colon \bar F \to G$ such that
    $s^\sharp \cdot r = s$ (see~\cite[Corollary~VII.2]{ablm} for
    details).
  \item Finally, note that the categories of coalgebras for $F$ and
    its Trnkov\'a hull $\bar F$ are clearly isomorphic.
  \end{enumerate}
\end{remark}

For the following fact, see e.g.~Ad\'amek et al.~\cite[Proof of
Lemma~8.8]{amm18}; we include the proof for the convenience of the
reader.
\begin{corollary}\label{cor:fin}
  The Trnkov\'a hull of a finitary set functor preserves all
  intersections.
\end{corollary}
%
\begin{proof}
  Let $F$ be a finitary set functor. Since $\bar F$ is finitary and
  preserves finite intersections, for every element $x \in \bar F X$,
  there exists a \emph{least} finite set $m\colon Y \subto X$ with $x$
  contained in $\bar F m$. Preservation of all intersections now
  follows easily: given subsets $v_i\colon V_i \subto X$, $i \in I$, with
  $x$ contained in the image of $\bar F v_i$ for each $i$, then $x$
  also lies in the image of the finite set $v_i \cap m$, hence
  $m \subseteq v_i$ by minimality. This proves
  $m \subseteq\bigcap_{i\in I} v_i$, thus, $x$ lies in the image of
  $\bar F(\bigcap_{i\in I} v_i)$, as required.
\end{proof}
\begin{remark}
  Note that the argument in \autoref{cor:fin} can be generalized to
  locally finitely presentable categories, see e.g.~Ad\'amek and
  Rosick\'y~\cite{AR94} for the definition.  In fact, let $\C$ be a
  locally finitely presentable category in which additionally every finitely
  generated object only has a finite number of subobjects; for
  example, $\Set$ or the categories of nominal sets (see
  \autoref{E:ass}\ref{E:ass:4}), of posets, and of graphs.

  Then every finitary endofunctor $F$ on $\C$ preserving finite
  intersections preserves all intersections. Indeed, since $F$ is
  finitary, for every monomorphism $m\colon Y \monoto FX$ with $Y$
  finitely generated there exists a subobject $z\colon Z \monoto X$
  with $Z$ finitely generated such that $m$ factorizes through $Fz$,
  i.e.~there exists some $g\colon Y \to FZ$ such that $Fm \cdot g = f$
  (see e.g.~\cite{amsw19}). Since $F$ preserves finite intersections
  it follows that there is a \emph{least} subobject
  $z\colon Z \monoto X$ such that $m$ factorizes through $Fz$. Indeed,
  let $z$ be the intersection of all subobjects $z'\colon Z \monoto X$
  such that $m$ factorizes through $Fz'$. This intersection is equal
  to the one of all $z'\cap z$, which is a finite intersection by our
  hypothesis since $Y$ is a finitely generated object. Since $F$
  preserves the latter finite intersection we obtain a morphism
  $g\colon Y \to FZ$ such that $Fz \cdot g = m$.  \smnote{One could
    say it even more detailed, but would have to make the pullback
    explicit, which would blow up the remark rather
    unpleasantly. Let's trust our readers to be able to do it
    themselves.}


  Preservation
  of all intersections now follows easily. Given subobjects
  $v_i\colon V_i \monoto X$, $i \in I$, and $m\colon Y \monoto FX$,
  with $m \leq Fv_i$ for all $i \in I$. We first assume that $Y$ is
  finitely generated. Take the least
  $z\colon Z \monoto X$ such that $m$ factorizes through $Fz$,
  i.e.~$m \leq Fz$. Then we have $m \leq F(v_i \cap z) \leq Fv_i$
  for all $i \in I$, where the first inclusion uses that $F$ preserves
  finite intersections. Furthermore, $m$ factorizes through $Fv_i$, and
  therefore $z \leq v_i$ by minimality for every $i \in I$. Thus,
  $m \leq Fz \leq F(\bigcap_{i\in I} v_i)$ as desired.

  For arbitrary $m\colon Y \to FX$ write $Y$ as the directed union of
  all its subobjects $s_j\colon Y_j \monoto Y$ with $Y_j$ finitely
  generated. Then every $s_j$ is contained in
  $F(\bigcap_{i\in I} v_i)$ by the previous argument, and therefore so
  is their union $m$.
\end{remark}

Let us conclude this section by coming back to coalgebras to note that
the condition that $F$ preserves intersection is significant for us
because it entails that every $F$-coalgebra has a \emph{reachable
  part}, i.e.~a unique reachable subcoalgebra. Indeed,
recall~\cite{amms13} that for an intersection preserving endofunctor
$F$ on a category $\C$ with intersections a reachable subcoalgebra can
be obtained as the intersection of all subcoalgebras of
$(C,c,i_C)$. Moreover, this intersection is the unique reachable
subcoalgebra of $(C,c,i_C)$. Given two reachable subcoalgebras
$S_1$ and $S_2$ of $(C,c,i_C)$ their intersection forms an $I$-pointed
subcoalgebra of $S_1$ and $S_2$ and so must be isomorphic to both,
thus $S_1 \cong S_2$.

\section{Canonical Graphs}
\label{S:cangraphs}
Note that for a given functor $F\colon \Set\to\Set$ one may define for every set
$X$ a map
\(
  \tau_X\colon FX\to \pow X
\)
by
\begin{equation}
  \hspace*{-5pt}
  \tau_X(t) = \{ x\in X\mid 1\xrightarrow{t} FX\text{ does not factorize through }
  F(X\setminus\{x\})\xrightarrow{Fi} FX \},
  \label{eq:tau}
\end{equation}
where $i\colon X\setminus\set{x} \subto X$ denotes the inclusion map.

Intuitively, $\tau_X(t)$ is the set of elements of $X$ that occur in $t$. 

\begin{definition}[Gumm~\cite{gumm05filter}]
  The \emph{canonical graph} of a coalgebra $c\colon C \to FC$ is the
  graph given by
  \[
    C \xrightarrow{c} FC \xrightarrow{\tau_C} \pow C.
  \]
\end{definition}
Note that for an $I$-pointed coalgebra $(C,c,i_C)$ its canonical graph
is $I$-pointed by $i_C\colon I \to C$, too.
\begin{example}
  For the functor $FX = \{0,1\} \times X^\Sigma$,
  we have for every $i \in \set{0,1}$ and $t\colon
  \Sigma \to X$ that
  \[
    \tau_X(i,t) = \{ t(s) \mid s \in \Sigma\}.
  \]
  Hence, the canonical graph of a deterministic automaton considered
  as an $F$-coalgebra is precisely its usual state transition graph
  (forgetting the labels of transitions and the finality of states).
\end{example}

\begin{lemma}[{Gumm~\cite[Theorem~7.4]{gumm05filter}}]\label{L:Gumm}
  If $F\colon \Set\to\Set$ preserves intersections, then the above
  maps $\tau_X\colon FX\to \pow X$ form a \emph{sub-cartesian
    transformation}, i.e.~for every injective map
  $m\colon X\rightarrowtail Y$ the following diagram is a pullback
  square: 
  \begin{equation}\label{diag:subcar}
    \begin{tikzcd}
      FX
      \arrow[>->]{d}[swap]{Fm}
      \arrow{r}{\tau_X}
      \pullbackangle{-45}
      & \pow X
      \arrow[>->]{d}{\pow m}
      \\
      FY
      \arrow{r}{\tau_Y}
      & \pow Y
    \end{tikzcd}
  \end{equation}
  Conversely, if $\tau$ is a sub-cartesian transformation, then $F$
  preserves intersections.\footnote{For this converse,
    Gumm assumed that $F$ preserves monomorphisms;
    however, this is not needed since $\pow$ preserves monomorphisms
    and monomorphisms are stable under pullback.}
\end{lemma}
%
%
\begin{theorem}[Gumm~{\cite[Theorem~8.1]{gumm05filter}}]
  Assume that $F$ preserves inverse images and intersections.
  Then $\tau\colon F\to \pow$ is a natural transformation.
\end{theorem}
\begin{example}\label{E:R}
  To see that $\tau$ is not a natural transformation in general, one
  may consider the functor $R\colon \Set \to \Set$ defined by $RX = \{(x,y)
    \in X \times X : x \neq y\} + \{*\}$ on sets $X$ and for a function $f\colon X
  \to Y$ put
  \[
    Rf(*) = * \quad\text{and}\quad
    Rf(x,y) = \begin{cases}
      * & \text{if $f(x) = f(y)$}\\
      (f(x),f(y)) & \text{else.}
    \end{cases}
  \]
  Now let  $X = \set{0,1}$, $Y = \set{0}$, and $f\colon X\to Y$ the evident function.
  Then $(0,1)\in RX$, and $\tau_X(0,1) = X$.  Furthermore,
  $\pow f(X) = Y$.  But $Rf(0,1) = *$, and $\tau_Y(*) = \emptyset$.
\end{example}

Our observation in this section is that reachability of a coalgebra
and its canonical graph are equivalent concepts:

\begin{theorem}\label{T:reach}
  Let $F\colon \Set \to \Set$ preserve intersections. Then an $I$-pointed
  coalgebra for $F$ is reachable if and only if so is its canonical
  graph.
\end{theorem}
\begin{proof}
  Let $(C,c,i_C)$ be an $I$-pointed $F$-coalgebra. Then we see that
  subcoalgebras of $(C,c,i_C)$ are in one-to-one correspondence with
  subgraphs of the canonical graph. Indeed, given any subcoalgebra
  $m\colon (S,s, i_S) \monoto (C,c,i_C)$, we have that
  $(S,\tau_S \cdot s,i_S)$ is an $I$-pointed subgraph of
  $(C,\tau_C \cdot c, i_C)$ via $m$ due to the commutativity
  of~\eqref{diag:subcar}. Conversely, let $(S,s,i_S)$ be an $I$-pointed subgraph of the
  canonical graph $(C, \tau_C \cdot c, i_C)$ via the monomorphism
  $m\colon S\monoto C$, say. Then, using that~\eqref{diag:subcar} is a
  pullback, we obtain an $F$-coalgebra structure on $S$ turning it
  into a subcoalgebra of $(C,c,i_C)$:
  \[
    \begin{tikzcd}
      S
      \arrow[>->,swap]{d}{m}
      \arrow[dashed]{r}
      \arrow[shiftarr={yshift=5mm}]{rr}{s}
      &
      FS
      \arrow{r}{\tau_S}
      \arrow[>->,swap]{d}{Fm}
      \pullbackangle{-45}
      &
      \pow S
      \arrow{d}{\pow m}
      \\
      C \arrow{r}{c}
      &
      FC \arrow{r}{\tau_C}
      &
      \pow C
    \end{tikzcd}
  \]
  We conclude that $(C,c,i_C)$ does not have any proper subcoalgebra
  w.r.t.~$F$ if and only if its canonical pointed graph
  $(C, \tau_C \cdot c, i_C)$ does not have a proper subcoalgebra
  w.r.t~$\pow$. As we saw in \autoref{E:reach}\ref{E:reach:1}, the
  latter is equivalent to that $I$-pointed graph being reachable, which
  completes the proof.
\end{proof}

\section{Iterative Construction}
\label{S:iter}

This section is devoted to a new iterative construction of the reachable part of
a given $I$-pointed coalgebra $(C, c, i_C)$, the unique reachable subcoalgebra
of $(C,c,i_C)$, reminiscent of breadth-first search for graphs.
\begin{assumption}\label{ass:main}
  Throughout this section we assume that the base category
  $\C$ has arbitrary (small)\sknote{small? SM: Yes, I think this goes without
  saying.} coproducts, is well-powered\smnote{I vote for
    making well-powered an assumption; it makes the whole setting
    nicer; see \autoref{R:ass}, we have no example where it does not
    hold, and we get the nice adjunction
    $\pasttime \dashv \nexttime$! JD: I concur}  and is equipped with an
  $(\E,\M)$-fac\-to\-ri\-za\-tion system, where $\M$ is a class of
  monomorphisms.
\end{assumption}
\begin{remark}\label{R:ass}\smnote{I would not turn this into
    propositions or lemmas. Those are no real results but just a
    couple of easy technical consequences of our
    assumptions that are probably well-known and that we simply
    collect for future use; perfect as a remark, imo.}
  We collect a number of easy consequence of \autoref{ass:main}. 
  \begin{enumerate}
  \item\label{R:ass:1} Note that the above assumptions imply that $\C$ has all
    unions, i.e.~for every object $C$ of $\C$ the partially ordered
    set $\sub(C)$ of its subobjects has all joins. In fact, given a
    family $(m_i\colon C_i \monoto C)_{i \in I}$, their union $m$ is
    given by the following $(\E,\M)$-factorization:
    \[
      \begin{tikzcd}
        \coprod_{i\in I} C_i
        \arrow[shiftarr={yshift=6mm}]{rr}{[m_i]_{i\in I}}
        \arrow[->>]{r}{e}
        & \bigcup_{i\in I} m_i
        \arrow[>->]{r}{m}
        & C.
      \end{tikzcd}
    \]

  \item It follows that $\sub(C)$ is a complete lattice, and therefore
    that $\C$ has all intersections.  Moreover, we show that
    intersections are given by pullbacks (even though we did not
    assume their existence). In fact,
    given the family $(m_i\colon C_i \monoto C)_{i \in I}$ take their
    intersection, i.e.~their meet, $m\colon M \monoto C$ in $\sub(C)$. The morphisms
    $p_i\colon M \monoto C_i$ witnessing $m \leq m_i$ yield the
    projections of the (wide) pullback. Moreover, given any compatible
    cone $f_i\colon X \to C_i$ such that
    $m_i \cdot f_i = m_j \cdot f_j$ for all $i,j \in I$ take the
    $(\E, \M)$-factorization $n \cdot e$ of that morphism and use the
    diagonal fill-in property
    \[
      \begin{tikzcd}
        X \arrow[->>]{r}{e} \arrow[swap]{d}{f_i}
        & X' \arrow[>->]{d}{n}
        \arrow[dashed]{ld} \\
        C_i \arrow[>->,swap]{r}{m_i} & C
      \end{tikzcd}
    \]
    in order to see that $n \leq m_i$ for all $i \in I$. Thus, we have
    $n \leq m$, which is witnessed by a (necessarily unique) morphism
    $h\colon X' \monoto M$ such that $n = m \cdot h$. Then $h \cdot e$
    is the desired unique factorizing morphism showing that $M$ is a
    wide pullback of the $m_i$.
    
  \item In addition, using the well-poweredness of $\C$ we see that it
    has \emph{preimages}, i.e.~pullbacks along morphisms in
    $\M$. \sknote{Isn't this an application of adjoint functor
      theorem? SM: Yes, you are right. But I do not think that making
      this explicit would simplify the current presentation.}
    Indeed,
    suppose we are given a morphism $f\colon X \to Y$ and a subobject
    $m\colon M \monoto Y$. Then we form the family of all subobjects
    $m_i\colon M_i \monoto X$ for which there exists a restricting
    morphism $f_i\colon M_i \to M$, i.e.~$f\cdot m_i = m \cdot f_i$,
    and we take their union:
    \begin{equation}\label{eq:u}
      (u\colon U \monoto X) 
      :=
      \bigcup \big\{m_i\colon M_i\monoto X \mid
      \text{$\exists f_i\colon M_i \to M$ with $f \cdot m_i = m\cdot
        f_i$}\big\}
    \end{equation}
    Using the diagonal fill-in
    property, we obtain a morphism $f'\colon U \to M$ such that the
    following diagram commutes:

    \[
      \begin{tikzcd}
        \coprod_{i \in I} M_i
        \arrow{rd}{[f_i]_{i\in I}}
        \arrow[->>,swap]{d}{e}
        \arrow[shiftarr={xshift=-10mm},swap]{dd}{[m_i]_{i \in I}}
        \\
        U
        \arrow[>->,swap]{d}{u}
        \arrow[dashed]{r}{f'}
        &
        M
        \arrow[>->]{d}{m}
        \\
        X \arrow{r}{f} & Y
      \end{tikzcd}
    \]
    In order to show that the lower square is a pullback, suppose that
    we have morphisms $p\colon Z \to X$ and $q\colon Z \to M$ with
    $f\cdot p = m \cdot q$. Take the $(\E,\M)$-factorization $p = (Z
    \stackrel{e'}{\epito} I \stackrel{m'}\monoto X)$. Then $(f \cdot
    m') \cdot e' = m \cdot q$. Hence, by the unique
    diagonal fill-in property, we obtain some $d\colon I \to M$ such
    that $m \cdot d = f \cdot m'$ and $d \cdot e' = q$. Thus, $m'\colon
    I \monoto X$ is one of the subobjects $m_i$ in~\eqref{eq:u}, and therefore $m'
    \leq \bigcup m_i = u$, i.e.~we have a morphism $s\colon I \to P$ with
    $u \cdot s = m'$. Then $h:= s \cdot e'\colon Z \to U$
    is the desired factorization of $p, q$. Indeed, we have
    \[
      u \cdot (s \cdot e') = m' \cdot e' = p,
    \]
    and to see that $f' \cdot h = q$ we use that $m$ is a
    monomorphism and compute
    \[
      m\cdot f' \cdot h = f \cdot u \cdot h = f \cdot p =
      m \cdot q.
    \]

  \end{enumerate}
\end{remark}
\begin{remark}\label{R:kleisli}
  In the following example and in \autoref{S:kleisli} we shall mention Kleisli
  categories. Recall that the Kleisli category $\Kl(T)$ for a monad
  $(T,\mu,\eta)$ on $\C$ has the same objects as $\C$ and hom-sets
  $\Kl(T)(X,Y) = \C(X,TX)$. We use the notation
  $f\colon X\kleislito Y$ to denote a morphism $f\in \Kl(T)(X,Y)$, and
  we call such morphisms \emph{Kleisli morphisms}. The composition of
  Kleisli morphisms $f\colon X\kleislito Y$ and
  $g\colon Y\kleislito Z$ is denoted by $g\circ f$ and defined by
  \[
    g\circ f = (X \xrightarrow{f} TY \xrightarrow{Tg} TTZ
    \xrightarrow{\mu_Z} TZ).\footnote{Note that in terms of the Kleisli
      extension $\ext g = \mu_Z \cdot Tg$ we have that $g \circ f =
      \ext g \cdot f$.}
  \]
  The identity morphism on $X$ is given by the unit $\eta_X\colon X\to
  TX$ of the monad.
\end{remark}
\begin{example}\label{E:ass}
  \begin{enumerate}
  \item\label{E:ass:1}
    Recall that every complete category $\C$ is equipped with a
    (strong epi, mono)-fact\-o\-ri\-za\-tion system and with an (epi,
    strong mono)-fac\-to\-ri\-za\-tion system~\cite[Theorems~14.17 and~14.19]{ahs09}.
    
    Hence, every complete and well-powered category $\C$ with
    coproducts meets \autoref{ass:main}.

  \item\label{E:ass:2}
    The category $\Rel$ of sets and relations has
    all coproducts and a factorization system given by
    \[
      \E = \text{all surjective relations,}
      \qquad\text{and}\qquad
      \M  = \text{all injective maps.}
    \]
    Note that $\Rel$ is the Kleisli category of the power-set monad.

  \item\label{E:ass:3} A similar factorization system can be obtained
    for stochastic relations which give for a point in a set $X$ a
    probability distribution over the points in another set $Y$ (in
    lieu of a set of points in $Y$). These stochastic relations are
    given by morphisms $X \to \Dist Y$ in the Kleisli category of the
    distribution monad $\Dist$ on $\Set$. This monad is given as a
    Kleisli triple $(\Dist, \eta, \ext{(-)})$ as follows: for every
    set $X$ we have
    \[
      \Dist X = \set{f\colon X \to [0,1] \mid \sum_{x \in X} f(x) = 1}
    \]
    (note that the above sum necessarily has at most countably many non-zero
    summands) and $\eta_X\colon X \to \Dist X$ given by the Dirac distribution
    \[
      \eta_X(x)(y) = \begin{cases}
        1 & \text{if }x = y, \\
        0 & \text{else}.
      \end{cases}
    \]
    The Kleisli extension of a map $h\colon X \to \Dist Y$ is the map
    $\ext h\colon \Dist X \to \Dist Y$ given by 
    \[
      \ext h(f)(y) = \sum_{x \in X} f(x) \cdot h(x)(y). 
    \]
    
    The Kleisli category $\Kl(\Dist)$ has all coproducts and a
    factorization system given by the following two classes of morphisms:
    \begin{align*}
      \E &= \set{ e\colon X \to \Dist Y \mid \forall y \in Y\, \exists x
      \in X\colon e(x)(y) \neq 0}, \text{and}\\
      \M & = \set{m\colon X \to \Dist Y \mid \text{$m = \eta_Y \cdot
          m'$ for some injective map $m'\colon X \to Y$}}.
    \end{align*}
    Hence, the class $\M$ consists essentially of injective maps
    considered as morphisms in $\Kl(\Dist)$. It is easy to see that the
    two classes of morphisms are closed under composition, and that
    every morphism in $\Kl(\Dist)$ has an essentially unique
    $(\E,\M)$-factorization, given by the least bounds of
    $\D$. Moreover, $\E$ and $\M$
    contain all isomorphisms of $\Kl(\Dist)$. In fact, we show below
    that those isomorphisms correspond precisely to bijective maps;
    more precisely, a morphism $h\colon X \to \Dist Y$ is an
    isomorphism if and only if $h = \eta_Y \cdot h'$ for some
    bijective map $h'\colon X \to Y$. It follows that
    $(\E,\M)$ is a factorization system~\cite[Theorem~14.7]{ahs09}.

    We proceed to prove the above characterization of isomorphisms in
    $\Kl(\Dist)$, i.e.~we show that $h\colon X \to \Dist(Y)$ is an
    isomorphism if and only if for every $x \in X$ there exists a
    $y \in Y$ with $h(x)(y) = 1$ and for every $y \in Y$ there is
    exists a unique $x \in X$ with $h(x)(y) > 0$ (and hence $h(x)(y) = 1$).

    The `if' direction holds, because the first condition states that $h =
    \eta_Y\cdot h'$ for some map $h'\colon X\to Y$ and the second condition
    states that $h'$ is bijective.

    For the `only if' direction suppose that
    $g\colon Y\to \Dist X$ is inverse to $h$. Observe that for every $x\in X$ there
    is some $y\in Y$ with $g(y)(x) > 0$, because $1 = (g\circ h)(x)(x) =
    \sum_{y\in Y} g(y)(x)\cdot h(x)(y)$, and similarly for every $y
    \in Y$ there is some $x \in X$ with $h(x)(y) >0$. By the
    definition of $\Dist X$, there is for every $y\in Y$ some $x\in X$ with $g(y)(x) \neq
    0$, and similarly, for every $x \in X$ there is some $y \in Y$
    with $h(x)(y) > 0$. 

    To verify the first condition, let $x\in X$ and $y_1,y_2\in Y$ with
    $h(x)(y_1)\neq 0 \neq h(x)(y_2)$. By the previous observation, there is some
    $y\in Y$ with $g(y)(x) > 0$, and therefore $(h \circ g)(y)(y_1) \neq 0 \neq
    (h \circ g)(y)(y_2)$. Hence, $y_1=y=y_2$ since $h\circ g$ is the identity
    morphism $\eta_Y$ in $\Kl(\Dist)$. So for every $x\in X$ there is
    a unique $y\in Y$ with $h(x)(y) > 0$ and therefore $h(x)(y) = 1$.

    For the second condition, we have already observed that for every $y\in Y$
    there exists some $x\in X$ with $h(x)(y) > 0$. For the uniqueness, let $y\in Y$,
    $x_1,x_2\in X$ with $h(x_1)(y)\neq 0 \neq h(x_2)(y)$. Let $x \in
    X$ be such that $g(y)(x) > 0$. Then $(g \circ
    h)(x_1)(x) \neq 0 \neq (g \circ h)(x_2)(x)$, and hence $x_1=x=x_2$ since
    $g\circ h$ is the identity morphism $\eta_X$ in $\Kl(\Dist)$.
    
  \item\label{E:ass:4} Let $\S = (S, +, \cdot, 0, 1)$ be a semiring. Then similarly
    as in point~\ref{E:ass:3} we obtain a monad $\S^{(-)}$ on $\Set$
    given as a Kleisli triple as follows: for every set $X$, we have
    \[
      \S^{(X)} = \set{f\colon X \to S \mid \text{$f(x) \neq 0$ for
          finitely many $x \in X$}},
    \]
    and the unit $\eta_X\colon X \to \S^{(X)}$ and the Kleisli
    lifting are defined precisely as in point~\ref{E:ass:3}.

    We also consider the same classes $\E$ and $\M$ as in the previous 
    point~\ref{E:ass:3}, and they can be shown to form a factorization
    system provided that the given semiring fulfills the following
    conditions: $(S, +, 0)$ and $(S, \cdot, 1)$ are positive monoids,
    i.e.~whenever $a + b = 0$ then $a = 0$ or $b= 0$, and similarly
    for the multiplication and $1$, and the semiring is
    zero-divisor-free, i.e.~whenever $a \cdot b = 0$ then $a = 0$ or
    $b = 0$.
  \end{enumerate}
\end{example}
Our construction of the reachable part is based on the following
notion capturing the part of an object $Y$ that is actually used by a
morphism $f\colon X \to FY$. For the class of all monomorphisms this
notion was introduced by Alwin Block~\cite{Block12} under the name
``base'':
\begin{definition}\label{def:leastfactor}
  Given a functor $F\colon \C\to \D$ and a class $\M$ of monomorphisms
  of $\C$. We say that $F$ \emph{has least bounds (w.r.t.~$\M$)} if
  for every morphism $f\colon X\to FY$ there is a \emph{least}
  morphism $m\colon Z\monoto Y$ in $\M$ such that $f$ factorizes through
  $Fm$. This means, there exists some $g\colon X\to FZ$ with
  \[
    \begin{tikzcd}
      X
      \arrow{r}{f}
      \arrow{dr}[swap]{g}
      & FY
      \\
      & FZ
      \arrow{u}[swap]{Fm}
    \end{tikzcd}
    \qquad\text{in $\D$},
  \]
  and for every $m'\colon Z'\to Y$ in $\M$ and 
  $g'\colon X\to FZ'$ with $Fm'\cdot g' = f$ there exists a
  (necessarily unique) $h\colon Z\to Z'$ with $m'\cdot h = m$,
  i.e.~$m \leq m'$ in $\sub(Y)$.

  The triple $(Z, g, m)$ is called the \emph{bound of $f$}, and the
  above triple $(Z', g', m')$ is said to \emph{compete} with the bound.
\end{definition}
\begin{proposition}\label{P:comp}
  Functors having least bounds are closed under composition.
\end{proposition}
\begin{proof}
  Let $F\colon \C \to \D$ and $G\colon \D \to \D'$ have least bounds
  w.r.t~the classes $\M$ and $\M'$ of $\C$ and $\D$, respectively. We
  will prove that $GF$ has least bounds w.r.t.~$\M$. In order to see
  this consider a morphism $f\colon X \to GFY$. Then first take its
  bound w.r.t.~$G$ to obtain $g\colon X \to GZ$ and
  $m\colon Z \monoto FY$ in $\M'$ such that $Gm\cdot g = f$, and then
  take the bound of $m$ w.r.t.~$F$ to obtain $g'\colon Z \to FZ'$ and
  $m'\colon Z'\monoto Y$ in $\M$ such that $Fm' \cdot g' = m$:
  \[
    \begin{tikzcd}
      X \arrow{rr}{f} \arrow[swap]{rd}{g}
      & &
      GFY
      \\
      &
      GZ \arrow[->]{ru}{Gm} \arrow[swap]{r}{Gg'}
      &
      GFZ'
      \arrow[->,swap]{u}{GFm'}
    \end{tikzcd}
  \]
  Then $Gg' \cdot g$ and $m'$ form the desired bound of $f$
  w.r.t.~$GF$. Indeed, given any $g''\colon X \to GFZ''$ and
  $m''\colon Z''\monoto Y$ with $GFm'' \cdot g'' = f$ one first uses
  minimality of the bound $(Z,g,m)$ to obtain some
  $h\colon Z \to FZ''$ with $Fm'' \cdot h = m$, and then one uses the
  minimality of $(Z',g',m')$ w.r.t.~$F$ to obtain $h'\colon Z'\to Z''$
  such that $m'' \cdot h' = m'$ as required.
\end{proof}
Gumm~\cite[Corollary~4.8]{gumm05filter} proved that, in the case
where $F$ is an endofunctor on $\Set$ and $\M$ is the class of all
monomorphisms, $F$ has least bounds if and only if it preserves
intersections. We now provide the proof in our setting, and we slightly extend the
result by a statement involving the following operator, which extends
the ``next time'' operator of Jacobs~\cite{Jacobs02} for coalgebras to
arbitrary morphisms:
\begin{definition}\label{D:nexttime}
  Suppose that $F\colon \C \to \C$ preserves $\M$-morphism, i.e.~$Fm$
  lies in $\M$ for every $m$ in $\M$. For every morphism
  $f\colon X \to FY$ we define the operator
  \[
    \nexttime_f\colon \sub(Y) \to \sub(X)
  \]
  as follows (we drop the subscript $f$ whenever this morphism is
  clear from the context): given a subobject $m\colon S \to Y$ we form the preimage
  of $Fm$ under $f$, i.e.~we form the pullback below: 
  \[
    \begin{tikzcd}
      \nexttime S
      \arrow[>->,swap]{d}{\nexttime m}
      \arrow{r}{f[m]}
      \pullbackangle{-45}
      &
      FS \arrow[>->]{d}{Fm}
      \\
      X \arrow{r}{f} & FY 
    \end{tikzcd}
  \]
\end{definition}
\begin{remark}
  Note that with our convention to take subobjects w.r.t.~the class
  $\M$, a functor $F\colon \C \to \C$ \emph{preserves
  intersections} if and only if it preserves $\M$-morphisms and (wide) pullbacks
  of families of $\M$-morphisms.
\end{remark}
\begin{proposition}\label{P:leastfact}
  Let $F\colon \C \to \C$ preserve $\M$-morphisms. Then the following
  are equivalent:
  \begin{enumerate}
  \item \label{P:leastfact:intsect} $F$ preserves intersections.
  \item \label{P:leastfact:least} $F$ has
    least bounds w.r.t.~$\M$. 
  \item\label{P:leastfact:3} For every $f\colon X \to FY$, the
    operator $\nexttime_f$ has a left-adjoint.
  \end{enumerate}
\end{proposition}
\noindent
Note that since $F$ preserves $\M$-morphisms, the factor $g$ in 
\autoref{def:leastfactor} is uniquely determined.
\begin{proof}
  For \ref{P:leastfact:3}$\implies$\ref{P:leastfact:intsect} choose
  $f = \id_{FY}$ then $\nexttime\colon m \mapsto Fm$ is a right-adjoint
  and so preserves all meets, i.e.~$F$ preserves intersections.

  The converse
  \ref{P:leastfact:intsect}$\implies$\ref{P:leastfact:3} 
  follows from the easily established fact that intersections are stable
  under preimage, i.e.~for every morphism $f\colon X \to Y$ and every
  family $m_i\colon S_i \monoto Y$ of subobjects the intersection
  $m\colon P \monoto X$ of the preimages of the $m_i$ under $f$ yields
  a pullback
  \[
    \begin{tikzcd}
      P
      \arrow[>->,swap]{d}{m} \arrow{r} \pullbackangle{-45} & \bigcap S_i
      \arrow[>->]{d}{\bigcap m_i}\\
      X \arrow{r}{f} & Y
    \end{tikzcd}
  \]
  Thus, if $F$ preserves intersections, so does $\nexttime$, whence it
  is a right-adjoint.
  
  For \ref{P:leastfact:intsect}$\implies$\ref{P:leastfact:least},
  consider $f\colon X\to FY$ and define $m\colon Y' \monoto Y$ to be
  the intersection of all subobjects with the desired factorization
  property:
  \begin{equation}\label{eq:int}
    (m\colon Y'\rightarrowtail Y) := \bigcap\big\{m_i\colon Y_i\rightarrowtail Y
    \mid \exists f_i\colon X\to FY_i\text{ with } Fm_i\cdot f_i = f\big\}
  \end{equation}
  This intersection exists since $\C$ is well-powered, and it is
  preserved by $F$. The witnessing morphisms $f_i\colon X\to FY_i$
  from~\eqref{eq:int} form a cone for this intersection, inducing a
  unique $f'\colon X\to FY'$ such that $Fs_i \cdot f' = f_i$ for all
  $i$, where $s_i\colon Y' \monoto Y_i$ are the morphisms witnessing
  $m \leq m_i$, i.e.~we have $m_i \cdot s_i = m$. It follows that we
  have
  \[
    Fm\cdot f' = Fm_i \cdot Fs_i \cdot f' = Fm_i \cdot f_i = f,
  \]
  whence $(Y', f', m)$ is the desired bound of $f$. In fact, minimality
  clearly holds: whenever $(\bar Y, g, \bar m)$ competes with that
  triple, we see that $\bar m$ is contained in the set
  in~\eqref{eq:int}, thus $m \leq \bar m$.

  For \ref{P:leastfact:least}$\implies$\ref{P:leastfact:intsect}, consider an
  intersection 
  \[
    (w\colon W \monoto Z)
    =
    \bigcap\big\{y_i\colon Y_i\rightarrowtail Z\mid i\in I \big\},
  \]
  and let $w_i\colon W\to Y_i$, $i\in I$, denote the corresponding
  pullback projections. Suppose we have a competing cone
  \[
    (c_i\colon C\to FY_i)_{i\in I}\text{ with } Fy_i\cdot c_i = Fy_j\cdot
    c_j\text{ for all }i,j\in I.
  \]
  We can assume wlog that $I\neq \emptyset$, because for
  $I=\emptyset$, the intersection $w=\id_Z$ is preserved by every
  functor.  We need to prove that there exists a unique morphism 
  $u\colon C\to FW$ such that $Fw_i\cdot u = c_i$ for all $i \in
  I$. Using our hypothesis~\ref{P:leastfact:least}, we take the bound $(Z', f', z)$
  of $f := Fy_i\cdot c_i$ for some $i \in I$. Hence, the following
  diagrams commute for all $i \in I$:
  
  \[
    \begin{tikzcd}
      & FY_i
      \arrow{d}{Fy_i}
      \\
      C
      \arrow{ur}{c_i}
      \arrow{dr}[swap]{f'}
      \arrow{r}{f}
      & FZ
      \\
      & FZ'
      \arrow{u}[swap]{Fz}
    \end{tikzcd}
  \]
  For every $i\in I$, the triple $(Y_i, c_i, y_i)$ competes with the
  bound of $f$. Hence, we have, for every $i\in I$, a unique
  $z_i\colon Z'\monoto Y_i$ with $y_i\cdot z_i=z$. Thus,
  $(z_i\colon Z'\monoto Y_i)_{i\in I}$ is a competing cone for the
  intersection of all $y_i$, and so we obtain a unique morphism
  $v\colon Z' \to W$ such that the following triangles commute:
  \[
    \begin{tikzcd}[sep = 10mm]
      & Y_i
      \\
      Z' \arrow[>->]{ur}{z_i}
      \arrow{r}{v}
      & W \arrow[>->]{u}[swap]{w_i}
    \end{tikzcd}
    \quad\text{for every $i\in I$}.
  \]

  Furthermore, since every $Fy_i$ lies in $\M$ and is therefore
  monomorphic, the following diagrams commute:
  \[
    \begin{tikzcd}
      FY_i
      \arrow[>->]{r}{Fy_i}
      & FZ
      \\
      C
      \arrow{u}{c_i}
      \arrow{r}{f'}
      & FZ'
      \arrow{u}[swap]{Fz}
      \arrow{ul}[description]{Fz_i}
    \end{tikzcd}
    \qquad
    \text{for every $i \in I$.}
  \]
  
  Now let $u := Fv\cdot f'\colon C\to FW$. Then the following diagram
  commutes: 
  \[
    \begin{tikzcd}[sep = 10mm]
      & & FY_i
      \\
      C
      \arrow[shiftarr={yshift=-6mm}]{rr}{u}
      \arrow[bend left=10]{urr}{c_i}
      \arrow{r}{f'}
      & FZ' \arrow{ur}{Fz_i}
      \arrow{r}{Fv}
      & FW \arrow[>->]{u}[swap]{Fw_i}
    \end{tikzcd}
    \quad\text{for every $i\in I$},
  \]
  as desired. Finally, since $Fw_i$ is monomorphic for every $i\in I$ and $I\neq
  \emptyset$, $u$ is the unique morphism such that $Fw_i \cdot u =
  c_i$ for every $i \in I$.
\end{proof}
\begin{assumption}
  \label{ass:main:F}
  In addition to \autoref{ass:main} we now assume for the remainder of
  this section that $F\colon \C \to \C$ is a functor preserving
  intersections (equivalently, $F$ has least bounds).
\end{assumption}
\takeout{
\begin{remark}
  Note that in the above proof well-poweredness of $\C$ is only used
  in the proof of
  \ref{P:leastfact:intsect}$\implies$\ref{P:leastfact:least}. Observe
  further that the equivalence of \ref{P:leastfact:intsect} and
  \ref{P:leastfact:least} only needs the existence of intersections,
  and the equivalence of \ref{P:leastfact:intsect} and
  \ref{P:leastfact:3} needs the existence of preimages, in addition.
\end{remark}}
\begin{remark}
  Note that if $F$ is a finitary functor on $\C = \Set$ with the usual
  factorization system given by surjective and injective maps, then we
  do not need to assume that $F$ preserves intersections. In fact,
  using \autoref{cor:fin}, we may work with the Trnkov\'a hull
  $\bar F$ recalling that the category of $\bar F$-coalgebras is
  isomorphic to the category of $F$-coalgebras.
\end{remark}
\begin{example} Let us continue \autoref{E:ass}.
  \begin{enumerate}
  \item Every intersection-preserving functor on a complete and
    well-powered category has least bounds. Note that this does not
    need the existence of coproducts; in fact, the proof of
    \autoref{P:leastfact} just needs the existence of intersections in
    $\C$.
  \item It is easy to see that every functor $\bar F$ on $\Rel$ extending
    an intersection-preserving set functor $F$ satisfies our assumptions,
    i.e.~$\bar F$ preserves maps and injective ones. Moreover, given a
    morphism $f\colon X \to FY$ in $\Rel$, we write it as a map
    $f\colon X \to \pow FY$. Using that $\pow\colon \Set \to \Set$
    preserves intersection and thus has least bounds, the
    argument of \autoref{P:comp} instantiated to show that $\pow F$
    has least bounds also shows how to obtain the bound of $f$
    w.r.t.~$\bar F$ on $\Rel$.

  \item A similar argument holds for extended functors $\bar F$ on
    $\Kl(\Dist)$ and $\Kl(\S^{(-)})$ for a semiring $\S$ satisfying the
    conditions mentioned in \autoref{E:ass}\ref{E:ass:4}. 

    However, we will see in \autoref{S:kleisli} that one
    can construct the reachable part of an $\bar F$-coalgebra even if
    one does not have a factorization system on the Kleisli
    category, i.e.~without any further assumption on the semiring
    $\S$. 
  \end{enumerate}
\end{example}
From the fact that $F$ has least bounds we obtain for every morphism 
$f\colon X \to FY$ the following operator mapping subobjects of $X$ to
those of $Y$: 
\begin{definition}
  Let $f\colon X \to FY$ be a morphism. The operator
  \[
    \pasttime_f\colon \sub(X) \to \sub(Y)
  \] 
  takes a subobject $m\colon S \monoto X$ to the bound of $f \cdot m$.  In
  particular, we have the commutative square below (we
  omit the subscript whenever $f$ is clear from the context):
  \[
    \begin{tikzcd}
      S \ar[>->,swap]{d}{m} \ar{r}{g} & F(\pasttime S) \ar[>->]{d}{F(\pasttime m)} \\
      X \ar{r}{f} & FY
    \end{tikzcd}
  \]
\end{definition}
\begin{proposition}\label{P:adj}
  For every morphism $f\colon X \to FY$, the operator $\pasttime$ is
  the left-adjoint of the ``next time'' operator $\nexttime$ from
  \autoref{D:nexttime}.

  Consequently, $\pasttime$ preserves all unions, and in particular it
  is monotone. 
\end{proposition}
\begin{proof}
  Let $f\colon X \to FY$ be any morphism and assume that
  $m \leq \nexttime m'$ for some subobjects $m\colon S \monoto X$ and
  $m'\colon S' \monoto Y$. Then we have a commutative diagram
  \[
    \begin{tikzcd}[sep = 10mm]
      S \arrow[>->]{r}{s} \arrow[>->,swap]{rd}{m}
      &
      \nexttime S' \arrow[>->]{d}{\nexttime m'} \arrow{r}{f[m']}
      \pullbackangle{-45}
      &
      FS'
      \arrow[>->]{d}{Fm'}
      \\
      &
      X \arrow{r}{f}
      &
      FY
    \end{tikzcd}
  \]
  and therefore $(S', f[m']\cdot s, m')$ is competing with the bound of
  $f \cdot m$. Thus, $\pasttime m \leq m'$. Conversely, suppose that
  $\pasttime m \leq m'$, witnessed by
  $j\colon \pasttime S \monoto S'$. Then consider the following
  diagram, where $g\colon S \to F(\pasttime S)$ comes from the bound
  of $f \cdot m$:
  \[
    \begin{tikzcd}
      S
      \arrow[>->, dashed]{d}
      \arrow{r}{g}
      \arrow[>->,shiftarr={xshift=-10mm},swap]{dd}{m}
      &
      F(\pasttime S)
      \arrow[>->]{d}{Fj}
      \arrow[>->,shiftarr={xshift=13mm}]{dd}{F(\pasttime m)}
      \\
      \nexttime S'
      \arrow{r}{f[m']}
      \arrow[>->,swap]{d}{\nexttime m'}
      \pullbackangle{-45}
      &
      FS'
      \arrow[>->]{d}{Fm'}
      \\
      X \arrow{r}{f} & FY
    \end{tikzcd}
  \]
  Since its outside and its right-hand part commute, we obtain the
  dashed arrow using the universal property of the pullback forming
  $\nexttime m'$. This proves that $m \leq \nexttime m'$ as desired.
\end{proof}
\begin{remark}
  For a coalgebra $c\colon C \to FC$, the operators $\nexttime$ and
  $\pasttime$ on $\sub(C)$ are a generalized semantic counterpart of
  the next time and previous time operators, respectively, of
  classical linear temporal logic, see e.g.~Manna and
  Pn\"ueli~\cite{MP92}. In fact, consider the functor
  $FX = \pow(A \times X)$ on $\Set$ whose coalgebras are labelled
  transition systems. For a transition system
  $c\colon C \to \pow(A \times C)$ and a subset $m\colon S \subto C$
  we have
  \begin{align*}
    \nexttime S & = \big\{x \in C \mid \text{for all $(a,s) \in c(x)$, $s \in S$}\big\},
    \\ 
    \pasttime S & = \big\{x \in C \mid \text{$(a,x) \in c(s)$ for some $a
      \in A$ and $s \in S$}\big\}.
  \end{align*}
\end{remark}
\begin{proposition}\label{P:ptcan}
  For an intersection preserving set functor, $\pasttime$
  may be computed on the canonical graph of a given coalgebra.
\end{proposition}
\takeout{
\begin{proof}
  Suppose that $F\colon \Set \to \Set$ (and by \autoref{ass:main:F}, $F$
  preserves intersections).  Let $c\colon C \to FC$ be a coalgebra and
  $s\colon S \subto C$ be any subset of states. We will show that
  $\pasttime s$ computed w.r.t.~$(C,c)$ or its canonical graph
  $(C,\tau_C \cdot c)$ agree.  Indeed, this follows from the fact that
  any triple $(Z, g, m)$ competing to the bound of $c\cdot s$ yields
  the triple $(Z, \tau_C \cdot g, m)$ competing to the bound of
  $\tau_C \cdot c \cdot s$, and moreoever, every triple $(Z, g', m)$
  competing to the latter bound must satisfy $g' = \tau_C \cdot g$
  where $(Z, g, m)$ competes to the former bound:
  \[
    \begin{tikzcd}
      S
      \arrow[>->,swap]{d}{s}
      \arrow[dashed]{r}{g}
      \arrow[shiftarr={yshift=5mm}]{rr}{g'}
      &
      FZ
      \arrow{r}{\tau_Z}
      \arrow[>->,swap]{d}{Fm}
      \pullbackangle{-45}
      &
      \pow Z
      \arrow{d}{\pow m}
      \\
      C \arrow{r}{c}
      &
      FC \arrow{r}{\tau_C}
      &
      \pow C
    \end{tikzcd}
  \]
  In fact, one uses that~\eqref{diag:subcar} is a pullback. Now if
  $(Z, g, m)$ is the bound of $c \cdot s$ then
  $(Z, \tau_Z \cdot g, m)$ is the bound of $\tau_C \cdot c \cdot s$,
  for if $(\hat Z, \hat g', \hat m)$ is any competing triple
  (w.r.t.~$\pow$) we see that $\hat g' = \tau_{\hat Z} \cdot \hat g$
  such that $(\hat Z, \hat g, \hat m)$ is a competing triple
  (w.r.t.~$F$), and thus $m \leq \hat m$. Conversely, if $(Z, g', m)$
  is the bound of $\tau_C \cdot m \cdot s$, then
  $g' = \tau_Z \cdot g$ so that $(Z, g, m)$ is the bound of
  $c \cdot s$, for if $(\hat Z, \hat g, \hat m)$ is a competing triple
  (w.r.t.~$F$), then $(Z, \tau_{\hat Z}\cdot \hat g, \hat m)$ is a
  competing triple (w.r.t.~$\pow$), and thus $m \leq \hat m$.
\end{proof}}
\begin{proof}
  Suppose $F\colon \Set \to \Set$, which preserves intersections by
  \autoref{ass:main:F}. By \autoref{L:Gumm} we have the sub-cartesian
  transformation $\tau_X\colon FX\to \pow X$ as defined
  in~\eqref{eq:tau}. Given a coalgebra $c\colon C \to FC$, we need to
  prove that $\pasttime_c = \pasttime_{\tau_C\cdot c}$. We know that
  $\pasttime_{c}$ is left-adjoint to $\nexttime_c$ by
  \autoref{P:adj}. Similarly, $\pasttime_{\tau_C\cdot c}$ is
  left-adjoint to $\nexttime_{\tau_C\cdot c}$, since $\pow$ 
  preserves intersections. Hence, it suffices to show that
  $\nexttime_c = \nexttime_{\tau_C\cdot c}\colon \sub(C)\to
  \sub(C)$. Indeed, for every $m\colon S\monoto C$ we have that the
  following composition of two pullbacks
  \[
    \begin{tikzcd}
      \nexttime_cS
      \arrow[>->,swap]{d}{\nexttime m}
      \arrow{r}{c[m]}
      \pullbackangle{-45}
      &
      FS
      \arrow{r}{\tau_S}
      \arrow[>->,swap]{d}{Fm}
      \pullbackangle{-45}
      &
      \pow S
      \arrow[>->]{d}{\pow m}
      \\
      C \arrow{r}{c}
      &
      FC \arrow{r}{\tau_C}
      &
      \pow C
    \end{tikzcd}
  \]
  which is again a pullback diagram. Thus
  $\nexttime m\colon \nexttime_cS \monoto C$ is the pullback of
  $\pow m$ along $\tau_C\cdot c$,
  i.e.~$\nexttime_c = \nexttime_{\tau_C\cdot c}$.
\end{proof}
\begin{remark}\label{R:adjrel}
  In connection with reachable coalgebras, the operator $\pasttime$
  was recently used in the work of Barlocco, Kupke, and
  Rot~\cite{BarloccoEA19}. Their results were obtained independently
  from ours but almost at the same time. They work with a complete and
  well-powered category $\C$, so $\M$ is the class of all
  monomorphisms (cf.~\autoref{E:ass}\ref{E:ass:1}). First they show
  that every intersection preserving endofunctor $F$ on $\C$ has least
  bounds (i.e.~the implication (1) $\implies$ (2) in
  \autoref{P:leastfact}).

  Furthermore, it is easy to see that, for every $F$-coalgebra
  $(C,c)$, $\pasttime$ is a monotone operator on $\sub(C)$. In
  addition, we see that $\pasttime$ preserves all unions. Indeed,
  in the setting of a complete and well-powered category, every
  $\sub(C)$ is a complete lattice having all intersections. Thus (the
  proof of) \autoref{P:leastfact} shows that $\nexttime$ has the
  left-adjoint $\pasttime$.

  Moreover, it is shown in op.~cit.~that for every point
  $i_0\colon 1 \to C$ the reachable part of $(C,c,i_0)$ is the least
  fixed point of $i_0 \vee \pasttime(-)$.

  However, note that the assumption of completeness may be limiting
  applications, e.g.~the category $\Rel$ in
  \autoref{E:ass}\ref{E:ass:2} is not complete.
\end{remark}
Barlocco et al.~\cite{BarloccoEA19} also prove the following fact. The
proof is the same in our setting, and we present it here for the
convenience of the reader.
\begin{proposition}
  Let $c\colon C \to FC$ be a coalgebra and $m\colon S \monoto C$ a
  subobject. Then $S$ carries a subcoalgebra of $(C,c)$ if and only if
  $m$ is a prefixed point of $\pasttime$. 
\end{proposition}
\begin{proof}
  Suppose first that we have $\pasttime m \leq m$, i.e.~we have some
  $i\colon \pasttime S \monoto S$ such that $m \cdot i = \pasttime
  m$. Then the following diagram commutes:
  \[
    \begin{tikzcd}
      S
      \arrow[>->,swap]{d}{m}
      \arrow{r}{g}
      &
      F(\pasttime S)
      \arrow[>->]{r}{Fi}
      \arrow[>->,swap]{rd}{F(\pasttime m)}
      &
      FS
      \arrow[>->]{d}{Fm}
      \\
      C \arrow[swap]{rr}{c} & & FC
    \end{tikzcd}
  \]
  This shows that $(S, Fi \cdot g)$ is a subcoalgebra of $(C,c)$. 

  Conversely, suppose that $m\colon (S,s) \monoto (C,c)$ is a
  subcoalgebra. Then $(S, s, m)$ competes with the bound $(\pasttime
  S, g, \pasttime m)$ of $c \cdot m$ and therefore we have $\pasttime
  m \leq m$. 
\end{proof}
We now present our new construction of the reachable part of an
$I$-pointed coalgebra as the union of all iterated applications of
$\pasttime$ on the given $I$-pointing.

\begin{construction} \label{constr}
  Given an $I$-pointed $F$-coalgebra
  $I\xrightarrow{i_C} C\xrightarrow{c} FC$, define subobjects
  $m_k\colon C_k\rightarrowtail C$, $k\in \N$, inductively:
  \begin{enumerate}
  \item Let $C_0$ be the $(\E,\M)$-factorization of $i_C$:
    \begin{equation}
      \begin{tikzcd}
        I
        \arrow[->>]{r}{i_C'}
        \arrow[shiftarr={yshift=6mm}]{rr}{i_C}
        & C_0
        \arrow[>->]{r}{m_0}
        & C.
      \end{tikzcd}
      \label{eq:constr:base}
    \end{equation}

  \item Given $m_k\colon C_k\to C$, let $m_{k+1} = \pasttime m_k\colon
    C_{k+1} = \pasttime C_k \rightarrowtail C$, i.e.~$m_{k+1}$ is
    given by the bound of  $c\cdot  m_k$:
    \begin{equation}\label{eq:constr:step}
      \begin{tikzcd}
        C_k
        \arrow{r}{c_k}
        \arrow[>->,swap]{d}{m_k}
        & FC_{k+1}
        \arrow[>->]{d}{Fm_{k+1}}
        \\
        C \arrow{r}{c}
        & FC
      \end{tikzcd}
    \end{equation}
  \end{enumerate}
  We define the subobject $m\colon R \monoto C$ to be the union of all
  $m_k$, $k \in \N$:
  \[
    m := \bigcup_{k \in \N} \pasttime^k(m_0).
  \]
\end{construction}
\begin{theorem}\label{constrCorrect}
  For every $I$-pointed coalgebra $(C,c,i_C)$, the union
  $m\colon R \monoto C$ is a reachable subcoalgebra of
  $(C,c,i_C)$.\jdnote{The statement is not complete \dots SM: better now?} 
\end{theorem}
\begin{proof}
  We will prove that $R$ carries the structure of an $I$-pointed
  coalgebra, which is reachable, and that $m$ is a homomorphism of
  $I$-pointed coalgebras.

  Recall from \autoref{R:ass}\ref{R:ass:1} that the union $m$ is
  formed by taking the $(\E,\M)$-factorization $m \cdot e = [m_k]_{k
    \in \N}\colon \coprod_{k \in \N} C_k \to C$, and let us denote the coproduct
  injections by $\inj_k\colon C_k \to \coprod_{k \in \N} C_k$.  
  
  By~\autoref{rem:EM}\ref{rem:EM:3},
  $e\cdot \inj_k\colon C_k\rightarrowtail R$ is in $\M$ for all
  $k \in \N$, because $m_k = m\cdot e \cdot \inj_k$ and $m$ is in $\M$.

  \begin{enumerate}
  \item The pointing $i_R := e\cdot \inj_0\cdot i_{C}'\colon I\to
    R$ is preserved by $m$ because the following diagram commutes
    \smnote{@Thorsten: can the arrow $i_C$ de set so that it leaves
      $I$ downwards and enters $C$ from the left with a 90 degree
      curve below $I$ and on the left of $C$? TW: Sure.} 
  \[
    \begin{tikzcd}
      I \arrow[->>]{r}{i_C'}
      \arrow[bend right=20,
      rounded corners,
      to path={ |- (\tikztotarget) \tikztonodes}
      ]{rrrd}[pos=0.75]{i_C}
      \arrow[shiftarr={yshift=5mm}]{rrr}{i_R}
      &
      C_0 \arrow{r}{\inj_0}\arrow[bend right=5]{rrd}[swap]{m_0}
      &
      \coprod_{k \in \N} C_k
      \arrow[->>]{r}{e} \arrow{rd}[description]{[m_k]_{k\in \N}}
      &
      R
      \arrow[>->]{d}{m}
      \\
      &&&
      C
    \end{tikzcd}
  \]

  For the coalgebra structure on $R$, first note that the
  outside of the following diagram commutes for every $\ell \in \N$ using
  \eqref{eq:constr:step}:
  \[
    \begin{tikzcd}[row sep=8mm, column sep=15mm]
      C_\ell
      \arrow{d}{\inj_{\ell}}
      \arrow[shiftarr={xshift=-10mm},>->]{dd}[swap]{m_{\ell}}
      \arrow{rr}{c_\ell}
      & & FC_{\ell+1}
      \arrow[shiftarr={xshift=10mm},>->]{dd}{Fm_{\ell+1}}
      \arrow{d}[swap]{F\inj_{\ell+1}}
      \\
      \coprod_{k\in \N} C_k
      \arrow{r}{\coprod_{k\in \N}c_k}
      \arrow{d}{[m_{k}]_{k\in \N}}
      & \coprod_{k \geq 1} FC_k
      \arrow{r}{[F\inj_{k+1}]_{k\in \N}}
      & F \coprod_{k\in \N} C_k
      \arrow{d}[swap]{F[m_{k}]_{k\in \N}}
      \\
      C
      \arrow{rr}{c}
      & & FC
    \end{tikzcd}
  \]
  Since the coproduct injections $\inj_\ell$ form a jointly epimorphic
  family, it follows that the lower part commutes, i.e.~we have
  obtained a coalgebra structure on $\coprod_{k \in \N} C_k$ such that
  $[m_k]_{k \in \N}$ is a coalgebra homomomorphism. Since
  $m\colon R\monoto C$ is obtained by the $(\E,\M)$-factorization
  $m \cdot e = [m_k]_{k \in \N}$, we also obtain a coalgebra structure
  $r\colon R \to FR$ such that $e$ and $m$ are $F$-coalgebra
  homomorphisms, using that $F$ preserves $\M$-morphisms and the
  unique diagonal fill-in property
  (cf.~\autoref{rem:subEM}\ref{rem:subEM:2}).

  \item For reachability, let $h\colon (S,s,i_S)
  \rightarrowtail (R, r, i_R)$ with $h\in \M$ be a pointed subcoalgebra.
  In the following we will define morphisms $d_k\colon C_k\rightarrowtail S$ (in
  $\M$) satisfying
  \begin{equation}\label{eq:dk}
    \begin{tikzcd}
      C_k
      \arrow[>->]{r}{m_k}
      \arrow[>->]{d}[swap]{d_k}
      & C
      \\
      S
      \arrow[>->]{r}{h}
      & R
      \arrow[>->]{u}[swap]{m}
    \end{tikzcd}
    \qquad\text{for all }k\in \N.
  \end{equation}
  We define $d_0\colon C_0\rightarrowtail S$ using the diagonal fill-in
  property; in fact, in the diagram below the outside commutes since $m_0
  \cdot i_{C'} = i_C$ is the pointing of $C$ and $m \cdot h$ preserves pointings:
  \[
    \begin{tikzcd}[column sep=6mm,row sep=12mm]
      I
      \arrow[->>]{rr}{i_C'}
      \arrow{d}[swap]{i_S}
      & & C_0
      \arrow[>->]{d}{m_0}
      \arrow[>->,dashed]{dll}[description]{d_0}
      \\
      S \arrow[>->]{r}{h}
      & R \arrow[>->]{r}{m}
      & C
    \end{tikzcd}
  \]
  Given $d_k\colon C_k\rightarrowtail S$, note that the following diagram
  commutes:
  \[
    \begin{tikzcd}
      C_k
      \arrow[>->]{r}{d_k}
      \arrow[>->]{dd}[swap]{m_k}
      &
      S \arrow{r}{s} \arrow[>->]{d}[swap]{h}
      &
      FS
      \arrow[>->]{d}{Fh}
      \arrow[>->,shiftarr={xshift=7mm}]{dd}{F(m \cdot h)}
      \\
      &
      R
      \arrow[>->]{ld}[swap]{m}
      \arrow{r}{r}
      &
      FR
      \arrow[>->]{d}{Fm}
      \\
      C \arrow{rr}{c} && FC
    \end{tikzcd}
  \]
  The commutativity of its outside means that
  $(S, s \cdot d_k, m \cdot h)$ competes with the bound
  $m_{k+1}\colon C_{k+1} \monoto C$ of $c \cdot m_k$. Thus, there
  exists a morphism $d_{k+1}\colon C_{k+1}\monoto S$ with
  $m\cdot h\cdot d_{k+1} = m_{k+1}$.

  Putting all squares~\eqref{eq:dk} together, we see that the diagram
  on the left below commutes: 
  \[
    \begin{tikzcd}[column sep = 14mm]
      \coprod_{k\in \N}C_k
      \arrow[bend right=10,shift right=2]{rr}[swap]{[m_k]_{k\in \N}}
      \arrow{d}[swap]{[d_k]_{k\in \N}}
      \arrow[->>]{r}{e}
      & R
      \arrow[>->]{r}{m}
      & C
      \\
      S
      \arrow[>->]{rr}{h}
      & & R
      \arrow[>->]{u}[swap]{m}
    \end{tikzcd}
    \overset{m\text{ monic }}{
    \Rightarrow
    }
    \begin{tikzcd}[column sep = 14mm]
      \coprod_{k\in \N}C_k
      \arrow{d}[swap]{[d_k]_{k\in \N}}
      \arrow[->>]{r}{e}
      & R
      \arrow[equals]{d}
      \arrow[dashed]{dl}[description]{\exists! d}
      \\
      S
      \arrow[>->]{r}{h}
      & R
    \end{tikzcd}
  \]
  Since $m\in \M$ is monomorphic the outside of the diagram on the
  right above commutes, and we apply the diagonal fill-in property
  again to see that $h$ is a split epimorphism. Since we also
  know that $h\in \M$ is a monomorphism, it is an isomorphism, whence
  $(R, r, i_R)$ is reachable as desired.\qedhere
\end{enumerate}
\end{proof}
\begin{definition}
  We call the above $I$-pointed coalgebra $(R,r,i_R)$ the
  \emph{reachable part} of $(C,c,i_C)$.
\end{definition}
\begin{remark}
  Note that it follows from an easy lattice theoretic argument that
  for every join-preserving map $\varphi\colon L \to L$ on a complete
  lattice $L$, and every $\ell \in L$ the least fixed point of
  $\ell \vee \varphi(-)$ is given by the join
  \begin{equation}\label{eq:join}
    \bigvee_{i \in \N} \varphi^i(\ell).
  \end{equation}
  Indeed, to see this recall that, by Kleene's fixed point theorem,
  the least fixed point of $\ell \vee \varphi(-)$ is the join
  of the $\omega$-chain given by $x_0 = \bot$, the least element of
  $L$, and $x_{n+1} = \ell \vee \varphi(x_n)$. We will show by
  induction that
  \[
    x_n = \bigvee_{i < n} \varphi^i(\ell)
    \qquad
    \text{for all $n < \omega$.}
  \]
  The base case $n=0$ is clear since the empty join is $\bot$, and for the
  induction step we use that $\varphi$ preserves joins to compute:
  \[
    x_{n+1}
    =
    \ell \vee \varphi(x_n)
    =
    \ell \vee \varphi\big(\bigvee_{i < n} \varphi^i(\ell)\big)
    =
    \ell \vee \bigvee_{i < n} \varphi^{i+1}(\ell)
    =
    \bigvee_{i < n+1} \varphi^i(\ell).
  \]
  
  Applying this to $\varphi = \pasttime$ we see that our
  \autoref{constr} of the reachable part coincides with the least
  fixed point of $i_0 \vee \pasttime(-)$ considered by Barlocco et
  al.~(cf.~\autoref{R:adjrel}).
\end{remark}

If we additionally assume that $F$ preserves inverse images, then the
reachability construction enjoys further strong properties:
\begin{theorem}\label{thm:coreflect}
  Suppose that $F$ preserves inverse images. Then the full subcategory
  of $\Coalg_I(F)$ given by all reachable coalgebras is coreflective.
\end{theorem}
%
%
\begin{proof}
  Let $(C,c,i_C)$ be an $I$-pointed $F$-coalgebra and let,
  $m\colon (R,r,i_R) \rightarrowtail (C,c,i_C)$ be its reachable
  part. We will show that this is a coreflection by verifying the
  corresponding universal property.

  Suppose we have a homomorphism $h\colon (S,s,i_S)\to (C,c,i_C)$
  where $(S,s,i_S)$ is reachable. We need to show that $h$ factorizes
  uniquely through $m$. Uniqueness is clear since $m$ is
  monic. For existence, we form the inverse image of $m$ under $h$,
  i.e.~we form the following pullback:
  \begin{equation}\label{diag:ref}
    \begin{tikzcd}
      P 
      \arrow{r}{h'}
      \arrow[>->]{d}[swap]{m'}
      \pullbackangle{-45}
      & R
      \arrow[>->]{d}{m}
      \\
      S
      \arrow{r}[swap]{h}
      & C
    \end{tikzcd}
  \end{equation}
  We equip $P$ with an $I$-pointing and a coalgebra structure making
  $m'$ and $h'$ homomorphisms of $I$-pointed coalgebras. Indeed, since
  $m\cdot i_R = i_C = h\cdot i_S$, we obtain a pointing $i_P$ on $P$,
  and since $F$ preserves inverse images we also obtain a coalgebra
  structure~$p$:
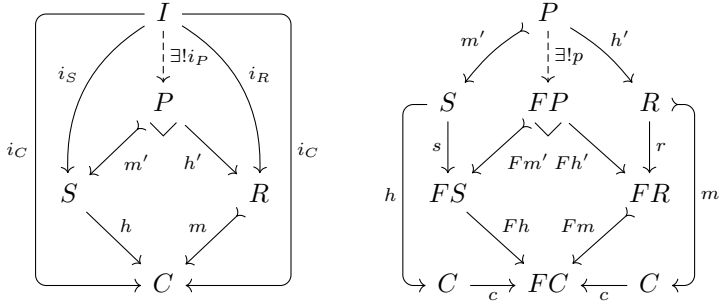
  \[
    \begin{tikzcd}[column sep=7mm,row sep=7mm]
      & I
      \arrow[bend right]{ddl}[swap]{i_S}
      \arrow[bend left]{ddr}{i_R}
      \arrow[dashed]{d}{\exists!i_P}
      \arrow[shiftarr={xshift=-17mm}]{ddd}[swap]{i_C}
      \arrow[shiftarr={xshift=17mm}]{ddd}{i_C}
      \\
      & P
      \arrow[shorten <=2pt, >->]{dl}{m'}
      \arrow[shorten <=2pt,]{dr}[swap]{h'}
      \pullbackangle{-90}
      \\
      S
      \arrow{dr}{h}
      & & R
      \arrow[>->]{dl}[swap]{m}
      \\
      & C
    \end{tikzcd}
    \qquad
    \begin{tikzcd}[column sep=5mm,row sep=7mm]
      & P 
      \arrow[bend right=10,>->]{dl}[swap]{m'}
      \arrow[bend left=10]{dr}{h'}
      \arrow[dashed]{d}{\exists! p}
      \\
      S
      \arrow{d}[swap]{s}
      \arrow[shiftarr={xshift=-6mm}]{dd}[swap]{h}
      & FP
      \arrow[shorten <=1pt, >->]{dl}{Fm'}
      \arrow[shorten <=1pt]{dr}[swap]{Fh'}
      \pullbackangle{-90}
      & R\arrow{d}{r}
      \arrow[shiftarr={xshift=6mm},>->]{dd}{m}
      \\
      FS
      \arrow{dr}{Fh}
      & & FR
      \arrow[>->]{dl}[swap]{Fm}
      \\
      C
      \arrow{r}[swap]{c}
      & FC
      & C
      \arrow{l}{c}
    \end{tikzcd}
  \]
  Clearly, this definition of $(P,p,i_P)$ makes $m'$ and $h'$
  homomorphisms. Since $(S,s,i_S)$ is reachable and $m'$ is
  monomorphic, the latter must be an isomorphism. Thus,
  $h'\cdot (m')^{-1}$ is the desired factorization of $h$ through $m$,
  cf.~\eqref{diag:ref}.
\end{proof}
\begin{corollary}\label{C:quot}
  If $F$ preserves inverse images, then reachable
  $F$-coalgebras are closed under quotients.
\end{corollary}
\begin{proof}
  Given a reachable coalgebra $(R,r,i_R)$ and $e\colon
  (R,r,i_R)\twoheadrightarrow (Q,q,i_Q)$ in $\Coalg_I(F)$ carried by an
  $\E$-morphism. Suppose that $m\colon (Q',q',i_Q') \to (Q,q,i_Q)$, $m\in \M$, is the
  inclusion of the reachable part of $Q$. By \autoref{thm:coreflect}, $e$
  factorizes through~$m$:
  \[
    \begin{tikzcd}
      (Q',q',i_Q') \arrow[>->]{r}{m}
      &
      (Q,q,i_Q)
      \\
      (R,r,i_R) \arrow[->>]{ru}[swap]{e}
      \arrow[dashed]{u}{h}
    \end{tikzcd}
  \]
  \sloppypar
  Using the diagonal fill-in property, we obtain a
  unique homomorphism $d\colon (Q,q,i_Q) \to (Q',q',i_Q')$
  such that $h = d \cdot e$ and $m\cdot d = \id$.
  This implies that $m$ is a split epimorphism, and since $m\in \M$ is
  a monomorphism, it is an isomorphism. 
\end{proof}
\begin{example}
  For functors not preserving inverse images, reachable coalgebras
  need not be closed under quotients. For example, recall the functor
  $R\colon \Set \to \Set$ from \autoref{E:R} and consider the
  coalgebras $c\colon C \to RC$ with $C = \set{x,y,z}$ and $c(x) = (y,z)$
  and $c(y) = c(z) = *$ and $d\colon D \to RD$ with $D =
  \set{x',y'}$ and $d(x') = d(y') = *$. Then $(D,d)$ is a quotient of
  $(C,c)$ via the coalgebra homomorphism $q$ with $q(x) = x'$ and
  $q(y) = q(z) = y'$. However, $(C,c, x)$ is reachable whereas
  $(D,d,x')$ is not.

  Note that, in the light of the proof of \autoref{C:quot}, this
  example also shows that reachable $F$-coalgebras need not form a
  coreflective subcategory if $F$ does not preserves inverse images.
\end{example}

Finally, let us come back to $\C = \Set$ and canonical graphs. We call
the subset $m_k\colon C_k \subto C$ of \autoref{constr} the \emph{$\kth$ step}
of the construction of the reachable part. 

\begin{corollary}\label{C:set}
  Let $F\colon \Set \to \Set$ preserve intersections. Then the
  $\kth$ steps of the constructions of the reachable parts of an
  $I$-pointed coalgebra and its canonical graph are the same.
\end{corollary}
Indeed, this follows by an easy induction from the fact that for a
coalgebra $c\colon C \to FC$ and a subset $s\colon S \subto C$,
$\pasttime s$ may be computed on the canonical graph of $(C,c)$ (see
\autoref{P:ptcan}).
\begin{remark}
  \begin{enumerate}
  \item From \autoref{C:set} we can conclude that for an intersection
    preserving set functor $F$, the reachable part of a given
    $I$-pointed $F$-coalgebra $(C,c,i_C)$ may be computed by a
    standard graph algorithm such as breadth-first-search. We thus
    obtain an efficient and generic algorithm for reachability of
    coalgebras for intersection preserving set functors.
    
  \item Moreover, the $\kth$ subset $C_k \subto C$ contains precisely
    all the states of $C$ that are reachable along a path of length
    precisely $k$ in the canonical graph of $C$ from a state in the
    image $C_0$ of the $I$-pointing $i_C\colon I \to C$.
  \end{enumerate}
\end{remark}

\section{Reachability in a Kleisli Category}
\label{S:kleisli}
In this section, we present a reduction from the reachability
construction in a Kleisli category for a monad on $\C$ to the
reachability construction in the base category $\C$. This makes our
construction applicable in Kleisli categories that fail to have an
$(\E,\M)$-factorization system for the desired class $\M$ of
monomorphisms that determines the notion of subcoalgebra. Coalgebras
over Kleisli categories are used to study the trace semantics of
various kinds of state-based systems, see e.g.~Hasuo, Jacobs, and
Sokolova~\cite{HJS07}.

For our reduction, we need that finite coproducts in $\C$ are well
behaved. In fact, recall~\cite{clw93} that a category is called
\emph{extensive} if it has finite coproducts and for every pair $A,B$
of objects the canonical functor $\C/A \times \C/B \to \C/(A+B)$ is an
equivalence of categories. 
\begin{remark}
  We further recall a few properties of extensive categories from~\cite{clw93}.
  \begin{enumerate}
  \item First note that a category with
    finite coproducts is extensive if and only if it has pullbacks
    along coproduct injections and for every
    commutative diagram
    \[
      \begin{tikzcd}
        A_1 \arrow{r}{a_1}
        \arrow{d}[swap]{h_1}
        & A
        \arrow{d}{h}
        & A_2 \arrow{l}[swap]{a_2}
        \arrow{d}{h_2}
        \\
        B
        \arrow{r}{\inl}
        & B+C
        & C
        \arrow{l}[swap]{\inr}
      \end{tikzcd}
    \]
    we have that $A_1 \xrightarrow{a_1} A \xleftarrow{a_2} A_2$ is a coproduct iff
    both squares are pullbacks.
      
  \item In an extensive category $\C$, the coproduct injections are
    monomorphisms, and coproducts are \emph{disjoint}, i.e.~the
    intersection of the subobjects $\inl\colon A\monoto A+B$ and
    $\inr\colon B\monoto A+B$ is $0 \monoto A + B$, where $0$ is the
    initial object.
  \end{enumerate}
\end{remark}
\begin{example}
  Many categories with set-like coproducts are extensive, for example
  the category $\Set$ itself, as well as the categories of partially
  ordered sets, nominal sets, and graphs as well as every category of
  presheaves. In addition, the categories of unary algebras and of
  J\'onsson-Tarski algebras (i.e.~algebras $A$ with one binary
  operation $A \times A \to A$ that is an isomorphism) are
  extensive. More generally, every topos is extensive.
  
  In contrast, the category of monoids is not extensive.
  \smnote{Do not say ``most algebraic categories''; this is an
    unjustified claim as you can quote nothing to substatiate it.}
\end{example}
Recall our terminology and notation for morphisms in a Kleisli
category for a monad from \autoref{R:kleisli}. Furthermore, recall
that a monad $(T, \mu, \eta)$ is called \emph{consistent} if $\eta_X$
is a monomorphism for every object $X$ of
$\C$~\cite[Definition~IV.2]{ablm}. Note that on $\Set$ all but two
monads are consistent. In fact, only the monad $C_1$ mapping all sets
to $1$ and $C_{01}$ mapping non-empty sets to $1$ and the empty set to
itself are inconsistent~(see e.g.~\cite[Lemma~IV.3]{ablm}).

The following terminology is borrowed from functional programming:
\begin{definition}
  A Kleisli morphism $f\colon X\kleislito Y$ is called \emph{pure} if
  $f\colon X\to TY$ factorizes through $\eta_Y\colon Y\to TY$. 

  A coalgebra homomorphism between coalgebras for a functor on
  $\Kl(T)$ which is pure is called \emph{pure coalgebra homomorphism}.
\end{definition}
\begin{remark}\smnote{Technical remark not part of definition!}
  The pure morphisms form a subcategory of $\Kl(T)$, and if $T$ is
  consistent, then this subcategory can be identified with the base
  category $\C$ via the canonical functor
  $J\colon \C \hookrightarrow \Kl(T)$ given by
  \[
     J(f\colon X\to Y) = (\eta_Z\cdot f)\colon X\kleislito Y.
  \]
  Consequently, we write $g\colon X\to Y$ for pure
  morphisms in diagrams in $\Kl(T)$ and omit the explicit application of the
  functor $J$.
\end{remark}
Recall that an \emph{extension} of a functor $F\colon \C\to \C$ to
$\Kl(T)$ is a functor $\bar F\colon \Kl(T)\to \Kl(T)$ such that the
square below commutes:
\[
  \begin{tikzcd}
    \Kl(T)
    \arrow{r}{\bar F}
    & \Kl(T)
    \\
    \C
    \arrow{r}{F}
    \arrow{u}{J}
    & \C
    \arrow{u}[swap]{J}
  \end{tikzcd}
\]
\noindent
If $T$ is consistent, then a functor on $\Kl(T)$ is an extension
iff it preserves pure morphisms.
\begin{example}
  \begin{enumerate}
  \item We have already mentioned that the category $\Rel$ is the Kleisli
    category of the power-set monad $\pow$. 
    Coalgebras over $\Kl(\pow)$ are systems with non-deterministic branching. For
    example, non-deterministic automata with input alphabet $\Sigma$ are
    coalgebras for $FX=1+\Sigma\times X$ (see
    \autoref{E:coalg}\ref{E:coalg:2}).
    
  \item For the finite power-set monad
    $\pow_f X = \{S\subseteq X\mid S\text{ finite}\}$, coalgebras on
    $\Kl(\pow_f)$ are finitely branching non-deterministic systems. For
    example, finitely branching transition systems with label alphabet
    $\Sigma$ are coalgebras for the functor $FX = \Sigma \times X$ on
    $\Kl(\pow_f)$.

  \item Consider the Kleisli category of the monad $\S^{(-)}$ given by a
    semiring $\S$ (see \autoref{E:ass}\ref{E:ass:4}). The $1$-pointed
    coalgebras for the functor $\bar FX = 1 + \Sigma \times X$ are weighted
    automata with the input alphabet $\Sigma$. Indeed, a coalgebra
    structure
    \[
      c\colon C \to \S^{(1 + \Sigma \times C)} \cong \S \times \S^{(\Sigma
        \times C)}
    \]
    assigns to each state $x$ an output value in $\S$, and for every
    $x \in X$ it yields a map $t_x\colon \Sigma \times C \to \S$ where
    $t_x(a,y) = s$ means that there is a transition from $x$ to $y$
    with label $a \in \Sigma$ and weight $s \in \S$.
    
  \item For the distribution monad $\Dist$, coalgebras on $\Kl(\Dist)$ have
    probabilistic branching.\smnote{I recalled the non-finite version
      of $\Dist$ in \autoref{E:ass}.} For example, coalgebras for the functor
    $\bar F X = \Sigma \times X$ are labelled Markov chains. Indeed, a
    coalgebra $c\colon C \to \Dist (\Sigma \times C)$ assigns to a state
    $x$ a distribution over pairs of labels and next states.  
  \end{enumerate}
  \takeout{
  There are many further examples of systems modelled as coalgebras over Kleisli
  categories, for example, conditional
  transition systems, which can be modelled as coalgebras over the Kleisli category
  of the reader monad on partially ordered sets~\cite{lmcsCTS}.
  \twnote{what else? JD: systems with final states using the maybe monad, 
  pushdown automata using the non-deterministic side effect monad (see Silva et al., 
  Generalizing determinization from automata to coalgebras)?}
}
\end{example}
\begin{assumption}
  We assume that $\C$ is an extensive category, that $(T,\mu,\eta)$ is a
  consistent monad on $\C$, and that $\bar F\colon \Kl(T)\to\Kl(T)$ is
  an extension of $F\colon \C\to \C$. Furthermore we consider the class
  $\M$ of all \emph{pure monomorphisms}:
  \[
    \M := \set{m\colon X\to Y\mid \text{$m$ is a monomorphism in $\C$}}.
  \]
\end{assumption}
\begin{remark}
    The notions of subobjects and of subcoalgebras in $\Kl(T)$
    are understood w.r.t.~the class $\M$, i.e.~a subobject is
    represented by a Kleisli morphism $m\colon S \kleislito X$ in $\M$ and a
    subcoalgebra by a coalgebra homomorphism $h\colon (S,s) \to (C,c)$
    carried by a Kleisli morphism in $\M$.
\end{remark}
\begin{notation}
  We write $\Coalg_I^\pure(\bar F)$ for the subcategory of
  $\Coalg_I(\bar F)$ given by pure coalgebra homomorphisms.
\end{notation}
\begin{construction}
  For every $I$-pointed coalgebra $I\overset{i_C}\kleislito
  C\overset{c}\kleislito \bar F C$ in $\Kl(T)$ we form the following $TF +
  T$-coalgebra in $\C$: 
  \jdnote{It is not completely clear what is the carrier from this 
  presentation. I guess it is $C+I$? SM: Reformulated, better now?}
  \[
    C+I \xrightarrow{c+i_C}
    TFC+TC
    \xrightarrow{TF\inl + T\inl}
    (TF+T)(C+I)
  \]
  together with the $I$-pointing $\inr\colon I \to C+I$. This defines
  the object assignment of a functor
  $G\colon \Coalg_I^\pure(\bar F) \to \Coalg_I(TF + T)$ that maps a
  pure coalgebra homomorphism $h\colon (C,c,i_C)\to (D,d,i_D)$ to
  $Gh = h+\id_I$.\footnote{The fact that $\eta$ is monic is used here
    to ensure that we may regard $h$ as a morphism of $\C$.}

  In fact, $Gh$ is a homomorphism of $I$-pointed coalgebras for $TF +
  T$ on $\C$ as shown by the following commutative diagram:
  \[
    \begin{tikzcd}
      I
      \arrow{r}{\inr}
      \arrow{dr}[swap]{\inr}
      & C+I
      \arrow{r}{c+i_C}
      \arrow{d}{h+I}
      & TFC+TC
      \arrow{d}{TFh+Th}
      \arrow{r}{TF\inl+T\inl}
      &[9mm] TF(C+I)+T(C+I)
      \arrow{d}{(TF+T)(h+I)}
      \\
      & D+I
      \arrow{r}{d+i_D}
      & TFD+TD
      \arrow{r}{TF\inl+T\inl}
      & TF(D+I)+T(D+I)
    \end{tikzcd}
  \]
\end{construction}
\begin{proposition}\label{kleisliCoalgExtensive}
  The functor $G$ reflects isomorphisms and it preserves and
  reflects subcoalgebras.
\end{proposition}
\begin{proof}
  Let $h\colon (C,c,i_C)\to (D,d,i_D)$ be a morphism in
  $\Coalg_I^\pure(\bar F)$. If $h$ is in $\M$, then it is a
  monomorphism in $\C$ and, moreover, so is $Gh = h + \id_I$ since
  monomorphisms are closed under coproducts in the extensive category
  $\C$. This shows that $G$ preserves subcoalgebras.

  To see that it reflects them assume that $h + \id_I$ is a
  monomorphism in $\C$. Then we have that $\inl \cdot h = (h + \id_I)\cdot
  \inl$ is monomorphic since $\inl$ is so. Thus $h$ is a monomorphism in $\C$, whence
  it is a pure monomorphism in $\Kl(T)$. 

  We proceed to proving that $G$ reflects isomorphisms. Consider
  $h\colon (C,c,i_C)\to (D,d,i_D)$ in $\Coalg_I^\pure(\bar F)$ such
  that $Gh = h + \id_I$ is an isomorphism is $\C$. By extensivity, we
  have the pullback
  \[
    \begin{tikzcd}
      C
      \arrow{d}[swap]{h}
      \arrow{r}{\inl}
      \pullbackangle{-45}
      & C + I
      \arrow{d}{h + I}
      \\
      D
      \arrow{r}{\inl}
      & D + I
    \end{tikzcd}
  \]
  Thus, $h$ is an isomorphism in $\C$, whence in $\Kl(T)$. 
\end{proof}
\begin{lemma}\label{L:kleisli2}
  \sloppypar
  Suppose that $T$ and $F$ preserve finite intersections, and let
  $h\colon (D,d,i_D)\to G(C,c,i_C)$ be a morphism in
  $\Coalg_I(TF + T)$ carried by a mono\-morphism in $\C$.
  Then there exists a morphism $g\colon (E,e,i_E) \to (C,c,i_C)$
  in $\Coalg_I^\pure(\bar F)$ such that $(D,d,i_D) = G(E, e, i_E)$ and
  $h = Gg$. 
\end{lemma}
\begin{proof}
  Consider a morphism $h$ in $\Coalg_I(TF+T)$ which is monomorphic in
  $\C$, i.e.~the following diagram commutes
  \begin{equation}\label{eq:h}
    \begin{tikzcd}
      I
      \arrow[>->]{r}{i_D}
      \arrow[>->]{dr}[swap]{\inr}
      & D
      \arrow{rr}{d}
      \arrow[>->]{d}{h}
      &
      &[9mm] TFD+TD
      \arrow[>->]{d}{TFh+Th}
      \\
      & C+I
      \arrow{r}{c+i_C}
      & TFC+TC
      \arrow{r}{TF\inl+T\inl}
      & TF(C+I)+T(C+I)
    \end{tikzcd}
  \end{equation}
  Form the pullbacks of the coproducts injections of $C+I$ along $h$:
  \begin{equation}\label{eq:pbs}
    \begin{tikzcd}
      E
      \arrow[>->]{r}{m}
      \arrow[>->]{d}[swap]{g}
      \pullbackangle{-45}
      & D
      \arrow[>->]{d}{h}
      & I
      \pullbackangle{-135}
      \arrow[>->]{l}[swap]{i_D}
      \arrow[equals]{d}
      \\
      C
      \arrow[>->]{r}{\inl}
      & C+I
      & I 
      \arrow[>->]{l}[swap]{\inr}
    \end{tikzcd}
  \end{equation}
  In order to see that the right-hand square is indeed a pullback,
  suppose we are given morphisms $p\colon X \to D$ and $q\colon X \to
  I$ such that $h \cdot p = \inr \cdot q$. Then we have
  \[
    h \cdot i_D \cdot q = \inr \cdot q = h \cdot p,
  \]
  and therefore we have $i_D \cdot q = p$ since $h$ is
  monomorphic. It follows that $q$ is the unique mediating morphism. 

  Thus, by extensivity, we have $D = E + I$ with coproduct injections $m$
  and $i_D$ and hence, $h = g + \id_I$. 

  Note that the two pullbacks in~\eqref{eq:pbs} are in fact
  intersections. Since $T$ and $F$ preserve finite intersections,
  the middle square in the next diagram is a pullback, and so is the
  right-hand one, by extensivity. By combining~\eqref{eq:h}
  and~\eqref{eq:pbs} we see that the outside of the following diagram
  commutes, and so we obtain the morphism $e\colon E \to TFE$ as indicated:
  \begin{equation}\label{eq:e}
    \begin{tikzcd}
      E
      \arrow[shiftarr={yshift=7mm}]{rrr}{d\cdot m}
      \arrow[>->]{d}[swap]{g}
      \arrow[dashed]{r}{\exists! e}
      & TFE
      \arrow{r}{TFm}
      \arrow[>->]{d}[swap]{TFg}
      \pullbackangle{-45}
      & TFD
      \arrow[>->]{d}[swap]{TFh}
      \arrow{r}{\inl}
      \pullbackangle{-45}
      & TFD+TD
      \arrow[>->]{d}{TFh+Th}
      \\
      C
      \arrow{r}{c}
      & TFC
      \arrow{r}{TF\inl}
      & TF(C+I)
      \arrow{r}{\inl}
      & TF(C+I) + T(C+I)
    \end{tikzcd}
  \end{equation}
  Similarly, we obtain $i_E$:
  \begin{equation}\label{eq:iE}
    \begin{tikzcd}
      I
      \arrow[shiftarr={yshift=7mm}]{rrr}{d\cdot i_D}
      \arrow{dr}[swap]{i_C}
      \arrow[dashed]{r}{\exists! i_E}
      & TE
      \arrow{r}{Tm}
      \arrow[>->]{d}[swap]{Tg}
      \pullbackangle{-45}
      & TD
      \arrow[>->]{d}[swap]{Th}
      \arrow{r}{\inr}
      \pullbackangle{-45}
      & TFD+TD
      \arrow[>->]{d}{TFh+Th}
      \\
      & TC
      \arrow{r}{T\inl}
      & T(C+I)
      \arrow{r}{\inr}
      & TF(C+I) + T(C+I)
    \end{tikzcd}
  \end{equation}
  Thus, we have seen that $g\colon (E,e,i_E)\to (C,c,i_C)$ is a
  morphism in $\Coalg_I^\pure(\bar F)$, i.e.~the diagram below
  commutes in $\Kl(T)$:
  \[
    \begin{tikzcd}
      I
      \arrow[kleisli]{r}{i_E}
      \arrow[kleisli]{dr}[swap]{i_C}
      & E
      \arrow[kleisli]{r}{e}
      \arrow{d}{g}
      & \bar F E
      \arrow{d}{\bar F g}
      \\
      & C
      \arrow[kleisli]{r}{c}
      & \bar F C
    \end{tikzcd}
  \]
  \sloppypar\noindent
  Finally, we establish that $G(E,e,i_E) = (D, d, i_D)$ by showing
  that the isomorphism $[m,i_D]$ (in $\C$) is a homomorphism from $G(E,e,i_E)$ to 
  $(D,d,i_D)$ in $\Coalg_I(TF+T)$:\smnote{Do not shorten
    ``homomorphism'' and ``from'' to get rid of sloppypar!}
  \[
    \begin{tikzcd}
      I
      \arrow{dr}[swap]{i_D}
      \arrow{r}{\inr}
      & E+I
      \arrow{r}{e+i_E}
      \arrow{d}{[m,i_D]}
      & TFE+TE
      \arrow{r}{TF\inl+T\inl}
      \arrow{dr}[description]{TFm+Tm}
      &[9mm] TF(E+I)+T(E+I)
      \arrow{d}{TF[m,i_D]+T[m,i_D]}
      \\
      & D
      \arrow{rr}{d}
      &
      & TFD+TD
    \end{tikzcd}
  \]
  Indeed, the two triangles trivially commute, and for the middle part
  consider the coproduct components separately: in fact, the left- and
  right-hand components are the upper parts of~\eqref{eq:e}
  and~\eqref{eq:iE}, respectively. This completes the proof. 
\end{proof}
\begin{corollary}\label{C:kreach}
  Let $T$ and $F$ preserve finite intersections.
  \begin{enumerate}
  \item The functor $G$ preserves and reflects reachable coalgebras. That is, 
    $G(C,c,i_C)$ is reachable iff so is $(C,c,i_C)$.
  \item The reachable part of an $I$-pointed $\bar F$-coalgebra $(C,c,i_C)$ is
    (up to isomorphism) given by the reachable part of $G(C,c,i_C)$.
    \jdnote{At this point, it is not clear what you mean here since both 
    coalgebras have different carriers. Only what is written after the 
    next remark makes it clearer. SM: I move that text up, ok now?}
  \end{enumerate}
\end{corollary}
\noindent
It follows that in order to construct the reachable part of an
$I$-pointed $\bar F$-coalgebra $(C,c,i_C)$ one may proceed as follows:
\begin{enumerate}
  \item Construct the reachable part of $G(C,c,i_C)$ in $\C$, and call the carrier $D$.
  \item Then $D = E+I$ for some subobject $m\colon E \monoto C$.
  \item $E$ carries an $I$-pointing $i_E\colon I \kleislito E$ and a
    $\bar F$-coalgebra structure $e\colon E \kleislito \bar F E $ such
    that $m\colon (E,e,i_E) \monoto (C,c,i_C)$ is the reachable part.
\end{enumerate}
Note that if $\Kl(T)$ and $\bar F$ fulfill \autoref{ass:main:F}, then
this gives the same result as performing \autoref{constr} directly on
$(C,c,i_C)$ in $\Kl(T)$ because the reachable
part of a coalgebra is unique up to isomorphism.
\begin{proof}
 \begin{enumerate}
 \item\label{C:kreach:1} For reflection, let $G(C,c,i_C)$ be reachable. By
   \autoref{kleisliCoalgExtensive}, every subcoalgebra
   $
     m\colon(S,s, i_S) \rightarrowtail (C,c,i_C)
   $
   is preserved by $G$, thus $Gm$ is an isomorphism, whence $m$ is one.
    
   For preservation, consider a subcoalgebra
   $m\colon (D,d,i_D)\monoto G(C,c,i_C)$. By \autoref{L:kleisli2},
   there exists $m'\colon (E,e,i_E) \to (C,c,i_C)$ in
   $\Coalg_I^\pure(\bar F)$ such that $m = Gm'$. Since $G$ reflects
   subcoalgebras by \autoref{kleisliCoalgExtensive}, $m'$ is a
   subcoalgebra. Finally, since $(C,c,i_C)$ is reachable, $m'$ is an
   isomorphism, thus so is $m$.
 \item This follows from point \ref{C:kreach:1} noting that
   $(C,c,i_C)$ has a unique reachable subcoalgebra since $G$ reflects
   isomorphisms by \autoref{kleisliCoalgExtensive}.\qedhere
%
  \end{enumerate}
\end{proof}
\begin{remark}
  Observe that in the case where the base category is $\Set$ one may
  drop the assumption that $T$ and $F$ preserve finite
  intersections. Indeed, if $I$ is the empty set, then the reachable
  part of every coalgebra is the empty subcoalgebra (in both $\C$ and
  $\Kl(T)$), and so the statement is trivial
  (cf.~\autoref{R:triv}). And if $I$ is non-empty, then the
  intersections computed in the above proof are non-empty and thus
  preserved by every $T$ and $F$, see Trnkov\'a~\cite{trnkova69}.
\end{remark}
%

\section{Conclusions and Future Work}

We have presented a new iterative construction of the reachable part
of a given $I$-pointed coalgebra. Our construction works for coalgebras for
intersection-preserving endofunctors over a category $\C$ which has
coproducts and a factorization system $(\E, \M)$ where $\M$ consists
of monomorphisms.
For coalgebras over $\Set$ we
saw that their reachable part can be constructed by running the 
standard breadth-first search algorithm on their canonical graph.  Finally, we
have considered coalgebras over Kleisli categories for a consistent
finite-intersection preserving monad $T$. We have shown that for a functor
$\bar F$ on $\Kl(T)$ extending a finite-intersection preserving functor
$F$ on $\C$, the reachable part of a given $I$-pointed coalgebra
can be obtained from the reachable part of an $I$-pointed
$(TF + T)$-coalgebra canonically constructed from the given one.

There remain a number of questions for future work. First, it should be
interesting to see whether our results still hold if we drop our assumption that
$\M$ is a class of monomorphisms. Secondly, it seems that the reachability
construction can be further generalized from working with (operators on)
subcoalgebras to working with fibrations, with the subobject fibration yielding
the present level of generality. Finally, we have seen that breadth-first search
is an instance of our reachability construction. A fibrational approach might provide other
breadth-first search based algorithms such as Dijkstra's algorithm for shortest
paths and Prim's algorithm for minimum spanning trees as special instances.

\bibliographystyle{plain}
\bibliography{refs}
\ifdraft{
\newpage
\tableofcontents
\begin{appendix}
  \input{appendix}
\end{appendix}
}{
}
\end{document}

%